\documentclass[reqno]{amsart}

\usepackage{amscd,amsthm,amsmath,amssymb,mathtools}
\usepackage{upgreek,bm}
\usepackage{verbatim}
\usepackage{color}
\usepackage{url}
\usepackage{enumerate,enumitem}
\usepackage{graphicx}
\usepackage{algorithm}
\usepackage{algorithmic}
\usepackage{epstopdf,epsfig,subfigure}
\usepackage{curve2e}
\usepackage{mathrsfs}
\usepackage{array}
\usepackage{stackrel}

\usepackage[numbers,sort&compress]{natbib}

\usepackage{setspace}
\usepackage[
    bookmarks=false,         
    unicode=true,          
    pdftoolbar=true,        
    pdfmenubar=true,        
    pdffitwindow=false,     
    pdfstartview={FitH},    
    pdftitle={Approximate controllability for 2D Euler equations},    
    pdfauthor={Author},     
    pdfsubject={Subject},   
    pdfcreator={Creator},   
    pdfproducer={Producer}, 
    pdfkeywords={keyword1, key2, key3}, 
    pdfnewwindow=true,      
    colorlinks=true,       
    linkcolor=red,          
    citecolor=blue,        
    filecolor=magenta,      
    urlcolor=cyan,           
hypertexnames=false
]{hyperref}

\theoremstyle{plain}

\newtheorem{theorem}{Theorem}[section]

\newtheorem{lemma}[theorem]{Lemma}
\newtheorem{corollary}[theorem]{Corollary}

\newtheorem{assumption}[theorem]{Assumption}

\theoremstyle{definition}
\newtheorem{definition}[theorem]{Definition}

\newtheorem{question}[theorem]{Question}

\numberwithin{equation}{section}

\usepackage{geometry}

\newcommand{\linspan}{\mathop{\rm span}\nolimits}

\newcommand{\supp}{\mathop{\rm supp}\nolimits}
\newcommand{\diver}{\mathop{\rm div}\nolimits}
\newcommand{\curl}{\mathop{\rm curl}\nolimits}

\newcommand{\rest}{\left.\kern-2\nulldelimiterspace\right|_}
\newcommand{\norm}[2]{\left|#1\right|_{#2}}
\newcommand{\dnorm}[2]{\left\|#1\right\|_{#2}}

\newcommand{\vol}{\mathop{\rm vol}\nolimits}

\newcommand{\Id}{{\mathbf1}}

\newcommand{\p}{\partial}

\newcommand*{\Bigcdot}{\raisebox{-.25ex}{\scalebox{1.25}{$\cdot$}}}

\newcommand{\clB}{{\mathcal B}}
\newcommand{\clC}{{\mathcal C}}
\newcommand{\clD}{{\mathcal D}}
\newcommand{\clE}{{\mathcal E}}

\newcommand{\clG}{{\mathcal G}}
\newcommand{\clH}{{\mathcal H}}

\newcommand{\clL}{{\mathcal L}}

\newcommand{\clP}{{\mathcal P}}

\newcommand{\clS}{{\mathcal S}}

\newcommand{\clU}{{\mathcal U}}

\newcommand{\clX}{{\mathcal X}}
\newcommand{\clY}{{\mathcal Y}}
\newcommand{\clZ}{{\mathcal Z}}

\newcommand{\bbB}{{\mathbb B}}

\newcommand{\bbE}{{\mathbb E}}
\newcommand{\bbF}{{\mathbb F}}

\newcommand{\bbN}{{\mathbb N}}

\newcommand{\bbR}{{\mathbb R}}
\newcommand{\bbS}{{\mathbb S}}
\newcommand{\bbT}{{\mathbb T}}

\newcommand{\bfB}{{\mathbf B}}

\newcommand{\bfH}{{\mathbf H}}

\newcommand{\bfX}{{\mathbf X}}

\newcommand{\fkB}{{\mathfrak B}}

\newcommand{\fkF}{{\mathfrak F}}

\newcommand{\fkH}{{\mathfrak H}}

\newcommand{\fkT}{{\mathfrak T}}
\newcommand{\fkU}{{\mathfrak U}}

\newcommand{\fkY}{{\mathfrak Y}}

\newcommand{\rmD}{{\mathrm D}}

\newcommand{\bfb}{{\mathbf b}}
\newcommand{\bfc}{{\mathbf c}}

\newcommand{\bff}{{\mathbf f}}

\newcommand{\bfn}{{\mathbf n}}

\newcommand{\rmc}{{\mathrm c}}
\newcommand{\rmd}{{\mathrm d}}
\newcommand{\rme}{{\mathrm e}}

\newcommand{\rms}{{\mathrm s}}

\newcommand{\fkc}{{\mathfrak c}}

\newcommand{\fkp}{{\mathfrak p}}

\newcommand{\fks}{{\mathfrak s}}

\newcommand{\fky}{{\mathfrak y}}

\newcommand{\tta}{{\mathtt a}}

\definecolor{DarkBlue}{rgb}{0,0.08,0.45}
\definecolor{DarkRed}{rgb}{.65,0,0}
\definecolor{applegreen}{rgb}{0.55, 0.71, 0.0}
\definecolor{gray06}{rgb}{0.6, 0.6, 0.6}

\newcounter{mymac@matlab}
\newcommand{\matlab}{MATLAB%
   \ifnum\value{mymac@matlab}<1%
   \textregistered%
   \setcounter{mymac@matlab}{1}%
   \fi%
  }

\newcommand{\red}{ \color{red} }

\usepackage{todonotes}
\newcommand{\todoautc}[3][]{%
    \ifthenelse{\equal{#1}{}}{\todo[size=\scriptsize]{{\bf#2} #3}{}}{\todo[color=#1,size=\scriptsize]{{\bf#2} #3}{}}%
}

\newcommand{\noteC}[1]{\todoautc[cyan!50]{}{#1}}

\newcommand{\addref}[1][]{{\red\tt addref}\noteC{...}}

\begin{document}
\title{Approximate controllability for 2D Euler equations}
\author{S\'ergio S.~Rodrigues$^{\tt1}$}

 \thanks{
\vspace{-1em}\newline\noindent
\newline\noindent
$^{\tt1}$ Johann Radon Institute for Computational and Applied Mathematics (RICAM), Austrian Academy of Sciences, Science Park 2,
  Altenbergerstrasse 69, 4040 Linz, Austria.
  \newline\noindent
{\sc Email}:
 {\small\tt sergio.rodrigues@ricam.oeaw.ac.at}.
  \newline\noindent
}

\begin{abstract}
Approximate controllability of the Euler equations is investigated by means of a finite set of
 actuators. 
It is proven that approximate controllability holds if we can find a  saturating subset of actuators. The notion of saturating set is relaxed when compared to previous literature, still being a sufficient condition for approximate controllability. 
The result holds for general bounded two-dimensional spatial domains with smooth boundary. An example of a saturating set is given in the case the spatial domain is the unit disk.
\end{abstract}
\maketitle

\pagestyle{myheadings} \thispagestyle{plain} \markboth{\sc S. S.
Rodrigues}{\sc Approximate controllability for 2D Euler equations}

%

\section{Introduction}
We investigate approximate controllability properties of the 2D Euler
system on bounded spatial domains~$\Omega\subset\bbR^2$.
The evolution is considered in a given time interval~$[0,T]$, $T>0$. The state of the system is the velocity field~$y(t,x)=(y_1(t,x),y_2(t,x)) \in\bbR^2$ of a given perfect incompressible fluid contained in~$\Omega$, with~$(t,x)\in(0,T)\times\Omega$, where~$x=(x_1,x_2)$ stands for the Cartesian coordinates of a generic point in~$\Omega$.

The controlled  dynamics is given by
\begin{align}
&\tfrac{\partial y}{\partial t}+ (y\cdot \nabla)y =f+
U^\diamond u+\nabla p_y,\qquad \diver y=0,\qquad \bfn\cdot y\rest{\p\Omega}=0,\qquad y(0,\cdot)=y_0,\label{eul-sys}
\end{align}
where~$\p\Omega$ denotes the boundary of~$\Omega$ and~$\bfn=\bfn(\overline x)$  stands for the unit outward normal vector to~$\p\Omega$, at~$\overline x\in\p\Omega$.
The nonlinear operator is understood as
\begin{equation}\label{B-formal}
(y\cdot \nabla)z\coloneqq(y\cdot\nabla z_1, y\cdot\nabla z_2).
\end{equation}
The function~$p_y(t,x)$ represents the pressure and~$f(t,x)$ is a given external forcing.

We are interested in the case where the control forcing~$U^\diamond u$ lies, for  every~$t\in(0,T)$, in a subspace~$\clU$ spanned by a finite number of appropriate linearly independent vector fields~$\Phi_k=\Phi_k(x)$,~$x\in\Omega$,
\begin{align}
&U\coloneqq \{\Phi_k\mid 1\le k\le M\}\subset L^{2}(\Omega)^2,\quad
\clU\coloneqq\linspan U,\quad \dim\clU=M,\notag\\
&U^\diamond\colon\bbR^M\to\clU,\quad U^\diamond z\coloneqq {\textstyle\sum_{k=1}^M} z_k\Phi_k,
\quad U^\diamond u(t)\coloneqq U^\diamond(u(t)),\notag
\end{align} 
where~$L^{2}(\Omega)$ denotes the Lebesgue space of square  integrable functions~$h\colon\Omega\to\bbR$ and the scalars~$z_k$ stand for the Cartesian coordinates of~$z=(z_1,z_2,\dots,z_M)\in\bbR^M$.
Thus, the control forcing~$U^\diamond u(t)\in\clU$ is realized via  the control input~$u=u(t)\in\bbR^M$, the choice of which
is at our disposal (e.g., among functions~$u\in L^\infty((0,T),\bbR^M)$).

The results apply to the case where the domain satisfies the following.
\begin{assumption} \label{A:domainMultiConn} 
$\Omega\subset\bbR^2$ is a bounded open connected subset with boundary $\p\Omega=\bigcup_{k=1}^{\varkappa}\Gamma_k$ given by the union of a finite number~$\varkappa$ of disjoint simple closed curves~$\Gamma_k$, $1\le k\le \varkappa$, each of class~$\clC^{\infty}$ and of finite length.
\end{assumption}

Multiconnected domains as in Assumption~\ref{A:domainMultiConn} are considered in several works in the literature related to Euler and Navier--Stokes equations; see~\cite{AuchAlex01}, \cite[Sect.~3]{FattoriniSrith92}, \cite{Kelliher06}, \cite[Appx.~I]{Temam01}. In particular, we shall make use of a result derived for such domains  in~\cite{Temam01}.

\subsection{Approximate controllability and saturating sets}\label{sS:appcontrSatset-intro}
Concerning approximate controllability, our goal is to find a set~$U$ of actuators so that for given arbitrary constants~$T>0$ and~$\varepsilon>0$, and arbitrary solenoidal vector functions~$y_0$ and~$y_\tta$, there exists a control input~$u$ so that the velocity field~$y$ solving~\eqref{eul-sys}, issued from~$y(0)=y_0$ at initial time~$t=0$,
satisfies~$\norm{y(T)-y_\tta}{L^2(\Omega)}<\varepsilon$ at final time~$t=T$.

Likely, the activity aimed at controlling Navier--Stokes and Euler systems by a finite number of actuators (independently of the viscosity coefficient~$\nu$ in the case of Navier--Stokes equations; see~\eqref{ns-sys})  has started in~\cite{EMat01,Romito04} addressing controllability properties of finite-dimensional Galerkin approximations.
In general, it is not trivial how to use controllability properties of such approximations to  derive analogue controllability properties for the infinite-dimensional (limit) Euler system.  
The study of the infinite-dimensional system started in~\cite{AgraSary05,AgraSary06}, where the spatial domain is the boundaryless Torus (i.e., periodic ``boundary conditions''; in this case the boundary condition~$\bfn\cdot y\rest{\p\Omega}=0$ in~\eqref{eul-sys} is void/omitted). The approach in~\cite{AgraSary05,AgraSary06} shows that the approximate controllability for 2D Euler and Navier--Stokes equations will follow if a suitable saturating subset~$G^0\subseteq U$ of actuators can be found. 

Usually, showing that a subset~$G^0$ of actuators is saturating is done by direct computations using properties of the  bilinear operator in~\eqref{B-formal}. Finding saturating sets turns out to be a nontrivial task, in particular, for bounded domains, due to the presence of the physical boundary~$\p\Omega$. After ``eliminating'' the pressure term, by writing the system as an evolutionary equation in the space of solenoidal vector fields
\begin{equation}\label{H}
H\coloneqq\{z\in L^2(\Omega)^2\mid\diver z=0\mbox{ and }\bfn\cdot z\rest{\p\Omega}=0\},
\end{equation}
we will be, in fact, interested in properties of the operator
\begin{equation}\label{B-Leray}
B(y,z)\coloneqq \Pi((y\cdot \nabla)z).
\end{equation}
where~$\Pi\colon L^2(\Omega)\to H$ stands for the Leray orthogonal projection in~$L^2(\Omega)$ onto~$H$. This projection depends on~$\Omega$ and computing~$B(y,z)$ explicitly is a nontrivial task, in general. 
In the case where~$\Omega=\bbT^2$ is a Torus~\cite{AgraSary05} it is possible to perform such computations and conclude the desired approximate controllability, by taking actuators given by eigenfunctions of the Stokes operator and by exploring the fact that such eigenfunctions have suitable qualitative properties as, for example,
\begin{equation}\label{Torus-B_eig}
\begin{split}
&\mbox{for two given eigenfunctions~$\phi$ and~$\psi$ in a Torus, we have that}\\
&\mbox{$B(\phi,\psi)$ is a linear combination of a finite number of eigenfunctions}
\end{split}
\end{equation}
and further, exploring the fact that we have simple explicit expressions available for such eigenfunctions. For general domains the explicit expressions for the eigenfunctions of the Stokes operator may  be not available or may not have the same qualitative properties (e.g., as~\eqref{Torus-B_eig}). Thus computing~$B(\phi,\psi)$ explicitly can be difficult (or, even not possible).

In~\cite[Def.~11]{AgraSary05} the elements of the saturating set are, by definition, eigenfunctions of the Stokes operator (thus, eigenfunctions of the Laplacian in vorticity formulation). In~\cite[Thm.~6.1]{AgraSary08} and~\cite[Def.~2.2.1]{Rod-Thesis08} (in the context of Navier--Stokes equations), it is proposed to take saturating sets consisting of more general steady states~$\psi$ of the unforced free Euler dynamics, thus satisfying~$B(\psi,\psi)=0$.

In~\cite{Shirikyan06,Shirikyan07}, in the context of Navier--Stokes equations, it is already required neither that the actuators~$\psi$ satisfy~$B(\psi,\psi)=0$ nor that~$\psi$ are eigenfunctions of the Stokes operator~$A$. It is however required that~$\psi$ are in the domain~$\rmD(A)$ of~$A$. This operator depends on suitable extra boundary conditions~$(\fkT y)\rest{\p\Omega}=0$, besides the~$\bfn\cdot y\rest{\p\Omega}=0$ in~\eqref{eul-sys}; see~\eqref{ns-sys} below. Since~$\fkT$ and~$A$ have nothing to do with the Euler equations, it makes sense to seek a ``relaxed'' notion of saturating set independently of these ``external'' operators~$\fkT$ and~$A$.

\begin{question}\label{Q:weaken}
Can we find a notion of saturating set of actuators involving only the boundary conditions~$\bfn\cdot y\rest{\p\Omega}=0$  in~\eqref{eul-sys} and, simultaneously, guarantee that the existence of such saturating set still implies approximate controllability? 
\end{question}

\subsection{Novelties}\label{sS:novel-intro}
The main goal of this manuscript is to give a positive answer to Question~\ref{Q:weaken}, in the context of spatial bounded domains~$\Omega\subset\bbR^2$. In the context of Euler equations, this is the first result in the literature for such domains, which necessarily have a nonempty boundary; previous works have considered only the case where the spatial domain is a boundaryless Riemannian manifold, namely, either a Torus~\cite{AgraSary05,AgraSary06} or  a Sphere~\cite{AgraSary08}. Further, to illustrate the relevance of the positive answer to Question~\ref{Q:weaken}, we give an example of a saturating set of actuators in the case the spatial domain is the unit disk, where the actuators do not satisfy classical boundary conditions for the Navier--Stokes equations as of Dirichlet (no-slip) or Lions (slip) type; in fact we do not know if they satisfy any extra boundary conditions, besides those in~\eqref{eul-sys}, which are meaningful for the Navier--Stokes systems. Recall also that, for bounded smooth domains~$\Omega\subset\bbR^2$ (with nonempty boundary), no example of a set of actuators  is known which satisfies other definitions of saturating sets in the literature involving the ``external'' Stokes operator~$A$ (examples are known for rectangular domains and under Lions boundary conditions only, in the context of Navier--Stokes equations~\cite{Rod06,PhanRod19}).

\subsection{Euler versus Navier--Stokes equations}
In this manuscript, we focus on the Euler equations, and postpone the investigation of the Navier--Stokes equations as
\begin{subequations}\label{ns-sys}
\begin{align}
&\tfrac{\partial y}{\partial t}-\nu\Delta y+ (y\cdot \nabla)y =f+
U^\diamond u+\nabla p_y,\quad\diver y=0,\\&\bfn\cdot y\rest{\p\Omega}=0,\quad (\fkT y)\rest{\p\Omega}=0,\quad y(0,\cdot)=y_0,
\end{align}
\end{subequations}
 for a future work; above~$\nu>0$ and~$\fkT$ is an operator imposing additional boundary conditions.

Though, at the first sight it is likely that an analogue procedure can be followed, there are key technical issues related with the different nature of the equations that must be carefully checked. These technical details become more evident when the actuators in the set~$U$ are either not regular enough or do not satisfy the boundary conditions. Some of these technical difficulties are avoided when the domain is boundaryless as in some previous literature.

 For the analysis it will likely be convenient to take a different functional setting. 
 In general, the well-posedness of the Euler equations  requires more smoothness for some data (as the initial velocity~$y_0$ and the external force $f+U^\diamond u$), while the analysis for Navier--Stokes equations will require the handling of the extra boundary conditions, $(\fkT y)\rest{\p\Omega}=0$. For example, no-slip Dirichlet conditions require the velocity to vanish on the boundary,~$y\rest{\p\Omega}=0$,  while slip Lions conditions require  the vorticity to vanish on the boundary, $(\curl y)\rest{\p\Omega}=0$. We shall shortly revisit these technical details associated with the boundary conditions, later in Section~\ref{sS:rmk-NS}.

\subsection{Related literature}\label{sS:liter-intro}
Without entering into the details at this point, the definition of saturating set will involve an increasing sequence of finite-dimensional subspaces~$\clG^{0}\coloneqq\linspan G_0$, $\clG^{j}\subset\clG^{j+1}\subset H$ constructed recursively by an appropriate procedure $\clG^{j}\mapsto\clG^{j+1}$ involving the nonlinear operator~\eqref{B-Leray}. The set~$G_0$ is said saturating if the union~$\bigcup_{j=0}^{\infty}\clG^j$  is dense in~$H$. 
So far, the examples of saturating sets available in the literature are given by subsets~$G_0$ consisting of eigenfunctions of the Stokes operator; see~\cite[Sect.~4, Cor.~2]{AgraSary06} for the case a Torus as spatial domain and~\cite[Thm.~10.4]{AgraSary08} of for the case of the 2D Sphere.

Though, we do not address the~3D case, we would like to mention the works in~\cite{Nersisyan10,Nersisyan11,Shirikyan08} (again, for the case of a Torus as spatial domain),  where  a variation of the Agrachev--Sarychev approach is used.
Also, though we do not address the Navier--Stokes equations, we would like to mention the examples of saturating sets as a subset of eigenfunctions given in~\cite[Def.~6.2]{Rod06} for 2D rectangles; in~\cite{Rod-Thesis08} for the 2D hemisphere;  in~\cite{PhanRod19} for 3D rectangles;  in~\cite{Phan21} for 3D cylinders. As a matter of fact, all these cases are somehow related to the cases of the Torus and Sphere.

Since it is not possible to guarantee, for general~$\Omega$, that the ``added'' vectors~$\clG^{j+1}\setminus\clG^{j}$ are (or contain) eigenfunctions, some relaxed definitions of saturating sets have been proposed in the literature. For example, in~\cite{Shirikyan06} for 3D Navier--Stokes equations including  no-slip boundary conditions, it is not required that the elements of the saturating sequence of vector fields~$\clG^j$ are the linear span of a finite number of eigenfunctions and it is again shown that the existence of such a relaxed saturating set also guarantees the approximate controllability.  The vector spaces~$\clG^j$ are only assumed to be subspaces of the domain~$\rmD(A)$ of the Stokes operator~$A\colon\rmD(A)\to H_3$ in~\cite[Sect.~2.1 (cf. Sects.1.1 and~1.2)]{Shirikyan06}, where~$H_3\coloneqq\{z\in L^2(\Omega)^3\mid\diver z=0\mbox{ and }\bfn\cdot z\rest{\p\Omega}=0\}$ is the subspace of solenoidal vector fields (cf.~\eqref{H}). An example of a saturating sequence satisfying such a constraint~$\clG^{j}\subset\rmD(A)$  is not known yet for no-slip boundary conditions. 
%
%
For further works where (variations of) this method, introduced in~\cite{AgraSary05}, is used we mention~\cite{Shirikyan07,Shirikyan06,PhanRod-ecc15,AgrKukSarShir07,Nersesyan21,Nersesyan15} for Navier--Stokes equations; \cite{Shirikyan18,Shirikyan_Evian07} for Burgers equations; \cite{Sarychev12} for Schr\"odinger  equations; \cite{NgoRaugel21} for second grade fluids; and~\cite{BoulvardGaoNers23} for primitive equations.

\subsection{Contents and notation}
In Section~\ref{S:defMainRes} we precise the concepts of saturating set and of approximate controllability and state the main results in Theorems~\ref{T:approxcontrol} and~\ref{T:saturating}. The proof of these results are presented in Sections~\ref{S:proofT:approxcontrol}--\ref{S:proofT:saturating-vort}. Section~\ref{S:final-rmks} collects concluding remarks on the derived result and on potential future work. Finally, the Appendix gathers the proofs of auxiliary results.

\medskip
Concerning notation, $\bbR$ will stand for the set of real numbers and~$\bbN$ for the set of nonnegative integers; their subsets of positive elements are denoted by~$\bbR_+\subset\bbR$ and~$\bbN_+\subset\bbN$.

Given  Hilbert spaces~$X$ and~$Y$, we denote by~$\clL(X,Y)$ the space of continuous linear mappings from~$X$ into~$Y$. The continuous dual of~$X$ is denoted~$X'\coloneqq\clL(X,\bbR)$.

By~$\Omega\subset\bbR^2$, we shall denote a generic bounded open and connected subset, which will be our spatial domain. For simplicity, we denote the Lebesgue function space~$L^2\coloneqq L^2(\Omega)$ and the Sobolev function spaces~$W^{s,2}\coloneqq W^{s,2}(\Omega)\subset L^2(\Omega)$, $s>0$. Again for simplicity, sometimes we shall denote spaces of vector fields as~$(L^2)^2=L^2\times L^2$ and~$(W^{s,2})^2=W^{s,2}\times W^{s,2}$,  simply by~$L^2$ and~$W^{s,2}$, respectively; the context will avoid ambiguities.

For~$k\in\bbN$, we denote by~$\clC^{k}=\clC^{k}(\overline\Omega)$ the space of functions from the closed subset~$\overline\Omega$ into~$\bbR$ with  bounded continuous partial derivatives up to order $k$, $k\in\bbN$. Similarly, ~$\clC^{k,1}=\clC^{k,1}(\overline\Omega)$ denotes the space of functions from~$\overline\Omega$ into~$\bbR$ with  bounded Lipschitz continuous partial derivatives up to order $k$.
Again, for vector fields, we denote~$\clC^{k}\times\clC^{k}$ and~$\clC^{k,1}\times\clC^{k,1}$, simply by
~$\clC^{k}$ and~$\clC^{k,1}$, respectively. We also denote~$\clC^{\infty}\coloneqq\bigcap_{k=0}^\infty\clC^{k}$.

 For functions depending on the time variable, we denote~$\clC^k([0,T],X)$ the space of functions $g\colon[0,T]\mapsto X$ with continuous derivatives~$(\frac{\rmd}{\rmd t})^j h$, ~$0\le j\le k$, and~$\clC^\infty([0,T],X)\coloneqq \bigcap_{k=0}^\infty\clC^k([0,T],X)$. Further, we denote~$\clC^{m;k,1}([0,T]\times\overline\Omega)$ the space of continuous functions~$h=h(t,x)$ from~$[0,T]\times\overline\Omega$ into~$\bbR$ such that~$(\frac{\p}{\p t})^j h\in\clC^0([0,T],\clC^{k,1}(\overline\Omega))$ is continuous from~$[0,T]$ into~$\clC^{k,1}(\overline\Omega)$ for all~$0\le j\le m$. Again, for vector fields, we denote
$\clC^{m;k,1}([0,T]\times\overline\Omega)\times\clC^{m;k,1}([0,T]\times\overline\Omega)$ simply by
$\clC^{m;k,1}([0,T]\times\overline\Omega)$.

\section{Definitions and main results}\label{S:defMainRes}
We introduce  spaces of  functions  which are appropriate to investigate the Euler system~\eqref{eul-sys}. We also present the definitions of saturating set and approximate controllability and, finally, we precise the statement of the main results. 

As usual we  consider the space~$L^2$ (both of functions and of vector functions) as a pivot Hilbert space, endowed with the usual scalar product. That is, it is identified with its continuous dual,~$L^2=(L^2)'$. Then, we can also identify~$H=H'$,
where~$H$ is as in~\eqref{H} and is endowed with the scalar product inherited from~$L^2$.

Hereafter, we assume that the initial state, the external forcing, and the set of actuators satisfy~$y_0\in H$, $f(t,\cdot)\in H$, and~$U\subset H$. Then, by using the Leray projection we 
 rewrite~\eqref{eul-sys} as an evolutionary equation in~$H$, as follows, with~$B$ as in~\eqref{B-Leray}, 
  \begin{align}\label{evol-sys-U}
&\dot y+ B(y,y) =f+
U^\diamond u,\qquad  y(0,\cdot)=y_0.
\end{align}

\begin{definition}\label{D:linspan}
Given a subset~$S\subset H$, the linear span of~$S$ is denoted 
\begin{align}\label{spanS}
\linspan S&\coloneqq\left\{z={\textstyle\sum_{j=1}^m}c_js_j\mid m\in\bbN_+\mbox{ and } (c_j,s_j)\in \bbR\times S\mbox{ for }1\le j\le m \right\}.
\end{align}
\end{definition}

We introduce the operator
  \begin{align}
\clB(z)y&\coloneqq B(z,y)+B(y,z)\notag
\end{align}
and, for subspaces~$Z\subseteq H$, $Y\subseteq H$, and~$\clH\subseteq H$, we also introduce the subspace
  \begin{align}\label{BZH-clH}
\clB_\clH(Z)Y\coloneqq\linspan \{\clH{\,\textstyle\bigcap\,}\clB(z)y\mid (z,y)\in Z\times Y \}.
\end{align}

\begin{definition}\label{D:saturVlin}
Given a subspace~$\clH\subset H$, a finite subset~$G_0\subset \clH$ is said $\clH$-saturating if the sequence~$(\clG^j)_{j\in\bbN}$ of subspaces~$\clG^j\subset \clH$ defined  by
\begin{align}\label{seqclG}
\clG^0&\coloneqq\linspan G_0,\qquad
\clG^{j+1}\coloneqq\clG^{j}+\clB_\clH(\clG^0)\clG^{j},
\end{align}
is such that the union~${\textstyle\bigcup_{j\in\bbN}}\clG^j$ is dense in~$H$. 
\end{definition}

\begin{definition}\label{D:acon-U}
Given~$T>0$, $ f\in L^\infty((0,T),H)$, and a subset~$U=\{\Phi_1,\Phi_2,\dots,\Phi_M\}\subset H$,  we say that  system~\eqref{evol-sys-U} is approximately controllable at time~$T$, if for every given initial state $y_0\in  H\bigcap \clC^{\infty}(\overline\Omega)$, every aim/target state~$y_\tta\in  H\bigcap\clC^\infty(\overline\Omega)$, and every~$\varepsilon>0$, we can find a control input~$u\in L^\infty((0,T),\bbR^M)$ such that the corresponding state satisfies
$\norm{y(T)-y_\tta}{H}\le \varepsilon$. 
\end{definition}

The first contribution of this manuscript is as follows. 

\begin{theorem}\label{T:approxcontrol}
Let~$G_0$ be a $(H\bigcap\clC^{2,1}(\overline\Omega))$-saturating set as in Definition~\ref{D:saturVlin-vort}, such that
  \begin{equation}\label{U}
U\coloneqq G_0{\,\textstyle\bigcup\,}\{B(h,h)\mid h\in G_0\}\subset H{\,\textstyle\bigcap\,}\clC^{2,1}(\overline\Omega).
\end{equation}
Let also $T>0$ and~$f\in \clC^{\infty;1,1}([0,T]\times\overline\Omega)\bigcap L^\infty((0,T),H)$. Then, system~\eqref{evol-sys-U} is approximately controllable at time~$T$.
Furthermore, we can take~$u\in \clC^\infty([0,T],\bbR^M)$. 
\end{theorem}
The proof is given in Section~\ref{S:proofT:approxcontrol}.
As we have mentioned above, this it is the first approximate controllability result for Euler equations concerning spatial domains as subsets~$\Omega\subset\bbR^2$, based on a saturating set as in Definition~\ref{D:saturVlin-vort}, where the recursion~\eqref{seqclG} does not require the elements of~$G_0$ to be regular equilibria of the unforced Euler equations~\cite[Thm.~6.1]{AgraSary08} or, more particularly, eigenfunctions of the vector Laplacian~\cite[Thm.~3]{AgraSary06}. The property of being an equilibrium and the high regularity of eigenfunctions can be helpful in easing/simplifying the proof of the result. Definition~\ref{D:saturVlin-vort} requires limited regularity of the actuators (intentionally avoiding requiring~$\clC^\infty$ regularity, again, in order to relax the definition of saturating set), thus not necessarily as regular as the eigenfunctions of the Laplacian; hence, we shall need to adapt some arguments. In particular, the intersection with~$\clH$, in~\eqref{BZH-clH}, is taken/needed to guarantee that the elements in each~$\clG^{j}$ are at least~``$\clH$-regular''; see~\eqref{seqclG}. This is because those elements will be used, within the proof, as auxiliary fictitious actuators, which must be regular enough to guarantee the existence and uniqueness of solutions under the corresponding auxiliary fictitious control forces.

Next, to illustrate the applicability of Theorem~\ref{T:approxcontrol}, we give an example of a saturating set in the case where the spatial domain is the unit disk
\[
\clD\coloneqq\{(x_1,x_2)\in\bbR^2\mid x_1^2+x_2^2<1\}.
\]
We consider, in polar coordinates, the set of functions
\begin{subequations}\label{satset-EG0}
\begin{align}\label{satset-vort}
&E\coloneqq\{r^{2k}\mid 0\le k\le 4\}{\,\textstyle\bigcup\,}\{r^{2(k+j)}\cos(k\theta),r^{2(k+j)}\sin(k\theta)\mid (j,k)\in\{0,1\}\!\times\!\{1,2,3\}\},\\
&\mbox{with}\quad (r,\theta)\in[0,1)\times[0,2\pi).\notag
\end{align}
thus with~$11$ elements. Then, we take a set of actuators including the vector fields the vorticities of which are the functions in~$E$, that is, now in Cartesian coordinates,
\begin{equation}\label{satset-vecfields}
G^0\coloneqq\{h\in (L^{2})^2\mid \bfn\cdot h\rest{\p\clD}=0\quad\mbox{and}\quad \curl h\in E\}.
\end{equation}
 \end{subequations}
The second contribution of this manuscript is the following. 
\begin{theorem}\label{T:saturating}
The set~$G_0$ in~\eqref{satset-EG0} is~$(H\bigcap\clC^{2,1}(\overline\clD))$-saturating and we have the inclusion~$G^0\bigcup\{B( h,h)\mid h\in G^0\}\subset H\bigcap\clC^{2,1}(\overline\clD)$.
\end{theorem}
The proof is given in Sections~\ref{S:proofT:saturating-vort-form} and~\ref{S:proofT:saturating-vort}. This is the first example in the literature concerning bounded smooth spatial domains~$\Omega\subset\bbR^2$ (with nonempty boundary). 

As a corollary it follows that~\eqref{evol-sys-U}, in the spatial domain~$\Omega=\clD$, is approximately controllable at time~$T$, by means of actuators in~$U=G^0\bigcup\{B( h,h)\mid h\in G^0\}\subset H\bigcap\clC^{2,1}(\overline\clD)$.

\section{Proof of Theorem~\ref{T:approxcontrol}}\label{S:proofT:approxcontrol}
 We present the proof in several steps.
We want to show approximate controllability by means of control forces taking values in the linear span of the set~$U$ of actuators~\eqref{U},
where~$G_0$ is a saturating set, leading to the increasing saturating sequence~$(\clG^j)_{j\in\bbN}$ as in Definition~\ref{D:saturVlin-vort}, with~$\clG^0=\linspan G_0$.
To simplify the exposition hereafter, we need to consider the more general form of~\eqref{evol-sys-U} as follows
\begin{align}\label{evol-sys-bfc}
&\dot y+ B(y,y) =f+\bfc, \qquad y(0,\cdot)=y_0,
\end{align}
 where~$\bfc$ is a general control forcing~$\bfc(t)\in H$. It is also convenient to introduce the following generalization of Definition~\ref{D:acon-U} allowing for a more general space~$\fkU$ of control forces.
\begin{definition}\label{D:acon-bfc}
Given~$T>0$, $ f\in L^\infty((0,T),H)$, and a subspace~$\fkU\subset L^\infty((0,T),H)$,  we say that system~\eqref{evol-sys-bfc} is $(T,\fkU)$-approximately controllable if for every given initial state~$y_0\in  \clC^{2,1}(\overline\Omega)\bigcap H$, every aim state~$y_\tta\in  \clC^{2,1}(\overline\Omega)\bigcap H$, and every~$\varepsilon>0$, we can find a control forcing~$\bfc\in \fkU$ such that the corresponding state satisfies
$\norm{y(T)-y_\tta}{H}\le \varepsilon$.
\end{definition}

We shall also use the following notation, for~$(m,k)\in\bbN\times\bbN$,
\begin{align}
     &\bfX_k\coloneqq\clC^{k,1}_H (\overline\Omega)\times  \clC^{1;k,1}_H([0,T]\times\overline\Omega)\times \clC^{0;k,1}_H ([0,T]\times\overline\Omega),\label{Xk}\\
     \mbox{with}\quad
    & \clC^{k,1}_H(\overline\Omega)\coloneqq H{\,\textstyle\bigcap\,}\clC^{k,1}(\overline\Omega),\notag\\  &\clC^{m;k,1}_H([0,T]\times\overline\Omega)\coloneqq  L^\infty((0,T),H){\,\textstyle\bigcap\,}  \clC^{m;k,1}([0,T]\times\overline\Omega).\notag
   \end{align}

Note that our goal (cf.~Thm.~\ref{T:approxcontrol}) is to prove that~\eqref{evol-sys-bfc} is $(T,\fkU)$-approximately controllable, for any apriori given~$T>0$, with~$\fkU= \clC^\infty([0,T],\clU)$, where~$\clU=\linspan U$. 
The proof is done in several  steps, within Sections~\ref{sS:proofACdensity}--\ref{sS:proofACiteration}, where a tuple
is assumed to be given, consisting of a time-horizon~$T>0$, an external force~$f$, an initial state~$y_0$, an aim state~$y_\tta$, and a constant~$\varepsilon>0$, satisfying
\begin{align}\label{fixedtuple}
&(T,f, y_0,y_\tta,\varepsilon)\in \bbR_+\times\clC^{\infty;2,1}_H([0,T]\times\overline\Omega)\times\clC^{2,1}_H(\overline\Omega)\times\clC^{2,1}_H(\overline\Omega)\times\bbR_+.
\end{align}

\subsection{Auxiliary results}\label{sS:proofACauxiliaty}
 We gather properties, for a class of systems extending~\eqref{evol-sys-bfc}, that we shall use later on. The class we consider is
  \begin{align}\label{evol-sys-bfcg}
&\dot y+ B(y+g_1,y+g_1) =g_2,\qquad y(0,\cdot)=\fky,
\end{align}
 with given functions~$g_1$ and~$g_2$. To simplify the exposition, the solution of~\eqref{evol-sys-bfcg} is denoted as~$y\eqqcolon\fkY(\fky;g_1,g_2)$, corresponding to a given tuple~$(\fky,g_1,g_2)$.

\subsubsection{Vorticity function}\label{ssS:vorticity} 
We shall use the following result from~\cite[Appx.~I, Prop.~1.4]{Temam01}, see also~\cite[Thm.~1]{AuchAlex01}.
\begin{lemma}\label{L:vect-vort-est}
Let~$\Omega$ be as in Assumption~\ref{A:domainMultiConn}. Then, the vorticity/curl operator~$y\mapsto \curl y\coloneqq-\frac{\p}{\p x_2}y_1+\frac{\p}{\p x_1}y_2$ maps~$H\bigcap W^{1,2}(\Omega)$ into~$L^2(\Omega)$ and
there exists a constant~$C>0$ such that
\begin{align}
\norm{y}{W^{1,2}(\Omega)}^2&\le C\left(\norm{y}{L^2(\Omega)}^2+\norm{\curl y}{L^2(\Omega)}^2\right),\quad\mbox{for all}\quad y\in H\,{\textstyle\bigcap}\, W^{1,2}(\Omega).\notag
\end{align}
\end{lemma}

\subsubsection{On existence and uniqueness of solutions}
\label{ssS:solutionsExiUni}
Recall the spaces~$\bfX_k$ as in~\eqref{Xk}. Solutions are understood as follows.
    \begin{theorem}\label{T:KatoShiri}
 Let~$\Omega$ be  as  in Assumption~\ref{A:domainMultiConn} and~$(\fky,0,g_2)\in \bfX_k$, $k\in\bbN_+$. Then, there exists one, and only one, solution~$y=\fkY(\fky;0,g_2)$ for~\eqref{evol-sys-bfcg} (with~$g_1=0$) such
that 
 \begin{align}\label{reg.EulerSol}
 \{y,\dot y,\p_{x_1}y,\p_{x_2}y\}\subset \clC^0([0,T],\clC^{k-1}(\overline\Omega)).
   \end{align}
 \end{theorem} 
This result follows from~\cite[Sects.~15.1--15.3]{EbinMars70} for general two-dimensional manifolds with either empty or smooth boundary.

Further, in case~$k=1$ the statement is also a straightforward corollary of the Theorem we find in~\cite[Intro.]{Kato67} and in case~$k>1$ the result can be shown by a bootstrapping-like argument as in~\cite[Sect.~2.3, proof of Thm.~2.7, step~3]{Shirikyan08-lect} (in there, shown for the case of simply-connected~$\Omega$).

We refer also to~\cite[Thm.~1]{BourguignonBrezis74}, including the case of  higher dimensional spatial domains~$\Omega\subset\bbR^d$, $d>2$, but with solutions defined locally in the temporal variable, that is, not necessarily in the given entire time interval~$[0,T]$. For additional literature, on solutions for the Euler equations, we refer the reader to~\cite{Yudovich95,Yudovich05,Constantin07,Wolibner33,LFilhoNLopesTadmor00,Temam75,ClopeauMikRob98} and references therein.

\begin{corollary}   
Given~$\Omega$ as in Assumption~\ref{A:domainMultiConn} and~$(\fky,g_1,g_2)\in  \bfX_2$, there exists one, and only one, solution~$y=\fkY(\fky;g_1,g_2)$ for~\eqref{evol-sys-bfcg} such
that $\{y,\dot y,\p_{x_1}y,\p_{x_2}y\}\subset \clC^0([0,T]\times\clC^1(\overline\Omega)).$
\end{corollary}  
\begin{proof} 
Let~$z\coloneqq y+g_1=\fkY(\fky+g_1(0); 0,g_2+\dot g_1)$. Note that~$\fky+g_1(0)\in\clC^{1,1}(\overline\Omega)^2$ and~$g_2+\dot g_1\in \clC^{0;2,1}([0,T],\clC^1(\overline\Omega))$. Then, by  Theorem~\ref{T:KatoShiri} we find
$\{z,\dot z,\p_{x_1}z,\p_{x_2}z\}\subset \clC^0([0,T],\clC^1(\overline\Omega)).$
Finally, observe that~$y=z-{g_1}$ and~$\{{g_1},\dot g_1,\p_{x_1}{g_1},\p_{x_2}{g_1}\}\subset \clC^0([0,T],\clC^1(\overline\Omega))$ . 
 \end{proof}

\subsubsection{The Stokes operator under Lions boundary conditions}\label{ssS:StokesLionsA}
We introduce the space
\begin{equation}\label{V}
 V\coloneqq H{\,\textstyle\bigcap\,}W^{1,2}
\end{equation}
and, with~$y(x)\eqqcolon(y_1(x),y_2(x))$ and~$x\eqqcolon (x_1,x_2)\in\Omega$, we introduce the~$\curl$ operator
\begin{equation}\label{curl}
\curl\colon V\to L^2, \qquad \curl y \coloneqq\frac{\p y_1}{\p x_2}-\frac{\p y_2}{\p x_1}.
\end{equation}
and the shifted Stokes operator~$A=\Id-\Pi\Delta$, under Lions boundary conditions, defined as
\begin{subequations}\label{StokesOperLions}
\begin{align}
&\langle A y,z\rangle_{V',V}\coloneqq(\curl  y,\curl \cdot z)_{L^2}+(y,z)_{H},\quad\mbox{for all}\quad  (y,z)\in V\times V,
\intertext{which we assume endowed with the scalar product}
& (y,z)_V\coloneqq\langle A y,z\rangle_{V',V},\qquad (y,z)\in V\times V.\label{scalarprodV}
\intertext{The domain of~$A$ is}
&\rmD(A)\coloneqq\{h\in H\mid Ah\in H\}=\{h\in H{\,\textstyle\bigcap\,}W^{2,2}\mid (\curl h)\rest{\p\Omega}=0\}.
\intertext{Since~$A^{-1}\colon H\to H$ is a compact operator, we can fix a complete system of orthonormal eigenfunctions~$e_k$ with associated eigenvalues~$\lambda_k$ (repeated accordingly to their multiplicity),}
&Ae_k=\lambda_ke_k,\qquad1\le\lambda_k\le\lambda_{k+1},\qquad \lim\limits_{k\to\infty}\lambda_k=\infty,\qquad k\in\bbN_+.\label{eigStokesLions}
\end{align}
\end{subequations}

Next, we can also define the fractional powers of~$A^{\gamma}$ as
\begin{align}
A^{\gamma}z\coloneqq{\textstyle \sum\limits_{k=1}^{\infty}}\lambda_k^{\gamma}z_k e_k,\notag
\end{align}
with domains~$\rmD(A^\gamma)$ endowed with the scalar product
\begin{align}
(z,w)_{\rmD(A^{\gamma})}={\textstyle \sum\limits_{k=1}^{\infty}}\lambda_k^{2\gamma}z_kw_k,\notag
\end{align}
for~$z=\sum_{k=1}^{\infty}z_ke_k$ and~$w=\sum_{k=1}^{\infty}w_ke_k$. In other words, $z\in\rmD(A^{\gamma})$ if, and only if, ~$A^{\gamma}z\in H$.

Note that~$H=\rmD(A^0)$, $V=\rmD(A^\frac12)$, $\rmD(A)=\rmD(A^1)$, and~$\rmD(A^\gamma)'=\rmD(A^{-\gamma})$.

\subsubsection{On the nonlinear term}\label{ssS:propB}
 For smooth functions, let us introduce the form
 \begin{align}\label{Trilin-poisson}
&\bfb(y,z,w)=\left\langle B(y,z),w\right\rangle\coloneqq(B(y,z),w)_H, \qquad \{y,z,w\}\subset C^1(\overline\Omega){\textstyle\bigcap H},
\end{align}
which can be extended continuously to larger spaces, as indicated by the following estimates.
\begin{lemma}\label{L:estTrilin-poisson}
The trilinear form~\eqref{Trilin-poisson} satisfies
\begin{align}
&\bfb(y,z,w)=-\bfb(y,w,z) \quad\mbox{and}\quad\bfb(y,z,z)=0,\quad\mbox{for all}\quad \{y,z,w\}\subset C^1(\overline\Omega)^2{\textstyle\bigcap H}.\label{asymTrilin}
\end{align}
Further, for a suitable constant~$C_\bfb>0$, we have the estimates
\begin{subequations}
\begin{align}
 \norm{\bfb(y,z,w)}{\bbR}=\norm{\bfb(y,w,z)}{\bbR}&\le C_\bfb\norm{y}{L^2}\norm{z}{\clC^{1}(\overline\Omega)}\norm{w}{L^2},\label{estTrilin-h.C1.h}\\
 \norm{\bfb(y,z,w)}{\bbR}=\norm{\bfb(y,w,z)}{\bbR}&\le C_\bfb\norm{y}{L^4}\norm{z}{V}\norm{w}{L^4},\label{estTrilin-l4.V.l4}\\
\norm{\bfb(y,z,w)+\bfb(z,y,w)}{\bbR}&\le C_\bfb \norm{(y,\nabla y)}{\clC^1}\norm{(z,\nabla y)}{\clC^1}\norm{w}{V'},\label{estTrilin-C1.C2.V'}
\end{align}
\end{subequations}
for all triples~$(y,z,w)\in H\times H\times H$ so that the corresponding right-hand side is bounded. In~\eqref{estTrilin-C1.C2.V'}, $V'$ is the continuous dual of~$V$, where the latter is endowed with the norm~$\norm{\Bigcdot}{V}$ associated with the scalar product in~\eqref{scalarprodV}.
\end{lemma}
The proof is given in the Appendix, Section~\ref{Apx:proofL:estTrilin-poisson}.   
   
\subsubsection{Particular basis for the subspaces~$\clG^j$}
\label{ssS:findimGj}
We shall use the following.
\begin{lemma}\label{L:findimGj} 
The subspaces~$\clG^j\subseteq\clH\subseteq H$ defined as in~\eqref{seqclG} are finite-dimensional. Furthermore, there exists a subset~$\{\Psi_k\mid 1\le k\le \overline m_{j}+n_j\}\subset \clG^{j+1}$ satisfying
\begin{align}\notag
&\clG^{j+1}=\linspan\{\Psi_k\mid 1\le k\le \overline m_{j}+n_j\},\quad \clG^{j}=\linspan\{\Psi_k\mid 1\le k\le \overline m_{j}\},\quad\dim\clG^{j}=\overline m_j,\notag\\
\intertext{with~$\overline m_{j}\in\bbN$ and~$n_j\in\bbN$. Further, if~$n_j\ge1$,}
&\hspace{2em}\Psi_{\overline m_{j}+i}=\clB(\psi_i)\phi_i,\quad\mbox{for}\quad 1\le i\le n_j,\quad\mbox{for some}\quad(\psi_i,\phi_i)\in\clG^0\times \clG^{j}.\notag
\end{align}
\end{lemma}
The proof is given in the Appendix, Section~\ref{Apx:proofL:findimGj}. 

\subsubsection{On time-average of functions weighted with oscillating functions}
\label{ssS:imitation-oscill}
The key step within the Agrachev--Sarychev approach concerns the ``imitation'' of control forces taking values in~$\clG^{j+1}$ by fast oscillating control forces taking values in~$\clG^{j}$; recall~\eqref{seqclG}. A key role within this step is played by the following result (cf.~\cite[Lem.~2]{AgraSary06}).
   
  \begin{lemma}\label{L:intsincosKTheta} 
 Given~$K\in\bbN_+$ and a scalar function~$\Theta\in W^{1,1}(0,T)$, we have
 \begin{subequations}
 \begin{align}
&\norm{{\int_0^T}\sin(\tfrac{\pi K s}{T})\Theta(s)\,\rmd s}{\bbR}\le \frac{(\pi C+1)T}{\pi}\norm{\Theta}{W^{1,1}(0,T)}K^{-1},\label{intsinKTheta}\\
&\norm{{\int_0^T}\cos(\tfrac{\pi K s}{T})\Theta(s)\,\rmd s}{\bbR}\le \frac{T}{\pi}\norm{\Theta}{W^{1,1}(0,T)}K^{-1},\label{intcosKTheta}
\end{align}
\end{subequations}
with~$C\coloneqq\norm{\Id}{\clL(W^{1,1}(0,T),L^\infty(0,T))}$.
  \end{lemma}
 The proof is given in the Appendix, Section~\ref{Apx:proofL:intsincosKTheta}. 

\subsubsection{On oblique projections}\label{ssS:obliqProj}
The first step within the Agrachev--Sarychev approach concerns the derivation of approximate controllability by means of control forces taking values in~$\clG^{\overline\jmath}$ for~$\overline\jmath$ large enough. In this step we shall follow a variation of the arguments in the literature, where we will make use of suitable  projections in the pivot Hilbert space~$H$. 
Let~$\clX$ and~$\clY$ be  two closed subspaces of~$H$ such that we have the direct sum~$H=\clX\oplus \clY$, that is, $H=\clX+\clY$ and~$\clX\bigcap\clY=\{0\}$. The (\emph{oblique}) projection onto~$\clX$ along~$\clY$ is denoted~$P_{\clX}^{\clY}$ and is defined by the identities
\begin{equation}\notag
P_{\clX}^{\clY}z\in \clX\quad\mbox{and}\quad z-P_{\clX}^{\clY}z\in \clY,\quad\mbox{for all}\quad z\in H.
\end{equation}
In particular, the \emph{orthogonal} projection in~$H$ onto~$\clX$ is given by~$P_{\clX}^{\clX^\perp}$, where
\begin{equation}
\clX^\perp=\clX^{\perp,H}\coloneqq\{h\in H\mid (h,s)_H=0\mbox{ for all }s\in \clX\}\notag
\end{equation}
denotes the complementary $H$-orthogonal space to~$\clX$.

\begin{lemma}\label{L:obliproj}
Let~$M\in\bbN_+$ and let~$\clE_M$ be the space spanned by the first eigenfunctions of the Stokes operator as in~\eqref{eigStokesLions}. Given~$\varrho>0$, we can find $\overline\jmath\in\bbN_+$ and~$\clU_M\subset\clG^{\overline\jmath}$ so that~$\norm{P_{\clE_M^\top}^{\clU_M}P_{\clE_M}^{\clE_M^\top}}{\clL(H)}<\varrho$.
  \end{lemma}
   The proof is given in the Appendix, Section~\ref{Apx:proofL:obliproj}.

\subsection{$(T,\fkU)$-approximately controllability with~$\fkU=\clC^\infty([0,T],\clU+\clG^{\overline\jmath})$, for large~${\overline\jmath}$} \label{sS:proofACdensity}
In this section, we use the density of the inclusion~$\bigcup_{j\in \bbN}\clG^{j}\subset H$.

Let~$(T,f,w_0,w_\tta,\varepsilon)$ be the tuple fixed in~\eqref{fixedtuple}.
Essentially, we shall be looking for fictitious/auxiliary actuators in the union~$\bigcup_{j\in \bbN}\clG^{j}$  which are close to the elements in~$\bbE_M\coloneqq\{e_k\mid 1\le k\le M\}$, where the~$e_k$ are the eigenfunctions in~\eqref{eigStokesLions}; we denote~$\clE_M\coloneqq\linspan\bbE_M$.

Let us fix~$\overline s\in(0,\frac12)$ and recall that~$f\in \clC^{\infty;2,1}_H([0,T]\times\overline\Omega)\subset L^2((0,T),W^{1,2})$. 
Observe that~$(w_0,w_\tta)\in \clC^{2,1}_H(\overline\Omega)\times  \clC^{2,1}_H(\overline\Omega)\subset V\times V\subset \rmD(A^{\frac{\overline s}2})\times\rmD(A^{\frac{\overline s}2})$. Observe also that by Lemma~\ref{L:vect-vort-est} and the relations~$\norm{y}{V}^2=\norm{y}{H}^2+\norm{(\curl y}{L^2}^2\le \norm{y}{H}^2+2\norm{\nabla y}{L^2}^2\le2\norm{y}{W^{1,2}}^2$, we have that~$\norm{\Bigcdot}{V}$ and~$\norm{\Bigcdot}{W^{1,2}}$ are equivalent norms in~$V=W^{1,2}\bigcap H$. Hence, we find
\begin{align}
&\norm{P_{\clE_M^\perp}^{\clE_M}z}{H}^2\le\lambda_{M+1}^{-\overline s}\norm{P_{\clE_M^\perp}^{\clE_M}z}{\rmD(A^{\frac{\overline s}2})}^2\le\lambda_{M+1}^{-\overline s}\norm{\Id}{\clL(W^{1,2},\rmD(A^{\frac{\overline s}2}))}^2\norm{z}{W^{1,2}}^2,\quad\mbox{for } z\in\{w_0,w_\tta\};\notag\\
&\norm{P_{\clE_M^\perp}^{\clE_M}f}{L^2((0,T),H)}^2\le\lambda_{M+1}^{-\overline s}\norm{\Id}{\clL(W^{1,2},\rmD(A^{\frac{\overline s}2}))}^2\norm{f}{L^2((0,T),W^{1,2})}^2.\notag
\end{align}

We can now choose~$M\in\bbN$ large enough so that 
\begin{align}\label{choiceM}
&\norm{P_{\clE_M^\perp}^{\clE_M}y_0}{H}^2<\frac{\varepsilon^2}{20}\rme^{-3T},\qquad\norm{P_{\clE_M^\perp}^{\clE_M}y_\tta}{H}^2<\frac{\varepsilon^2}{20},\qquad
\norm{P_{\clE_M^\perp}^{\clE_M}f}{L^2((0,T),H)}^2<\frac{\varepsilon^2}{20}\rme^{-3T}.
\end{align}

Next, we consider an arbitrary space~$\widehat\clU_M$ such that we have the direct sum~$H=\widehat\clU_M\oplus \clE_M^\top$. We use the oblique projection~$P_{\widehat\clU_M}^{\clE_M^\top}$ to construct the control~$c(t)\in\widehat\clU_M$ as
\begin{align}\label{control-cM}
c\coloneqq\dot\eta-P_{\widehat\clU_M}^{\clE_M^\top}f,\quad\mbox{with}\quad\eta(t)\coloneqq \frac1T P_{\widehat\clU_M}^{\clE_M^\top}\left((T-t)\rme^{-\mu t}w_0+t\rme^{-\mu (T- t)}w_\tta\right),
\end{align}
with~$\mu>0$. Let~$y$ be the corresponding state, thus solving
\begin{align}\notag
&\dot y +B(y,y)=f+\dot \eta- P_{\widehat\clU_M}^{\clE_M^\top}f,\qquad y(0)=y_0.
\end{align}
Observe that~$z\coloneqq y-\eta$ satisfies
\begin{align}\notag
&\dot z+ B(z+\eta,z+\eta) =P_{\clE_M^\top}^{\widehat\clU_M}f,\qquad z(0)=P_{\clE_M^\top}^{\widehat\clU_M}y_0,
\end{align}
from which we also find
\begin{align}
\tfrac{\rmd}{\rmd t}\norm{z}{H}^2&=2\bfb(z+\eta,\eta,z)+2(P_{\clE_M^\top}^{\widehat\clU_M}f,z)_H\notag\\
&\le2C_\bfb\norm{z+\eta}{L^2}\norm{\eta}{\clC^{1}(\overline\Omega)}\norm{z}{L^2}+2\norm{P_{\clE_M^\top}^{\widehat\clU_M}f}{H}\norm{z}{H}\notag\\
&\le (2+2C_\bfb\norm{\eta}{\clC^{1}(\overline\Omega)})\norm{z}{L^2}^2+C_\bfb^2\norm{\eta}{L^2}^2\norm{\eta}{\clC^{1}(\overline\Omega)}^2+\norm{P_{\clE_M^\top}^{\widehat\clU_M}f}{H}^2,\notag
\end{align}
where we used~\eqref{asymTrilin} and~\eqref{estTrilin-h.C1.h}.
By the Gronwall inequality
\begin{align}
&\norm{z}{L^\infty((0,T),H)}^2\le \rme^{\xi_1}\left(\norm{z(0)}{H}^2 +\xi_0\right),\label{bddQ}\\
\mbox{with}\quad &\xi_1\coloneqq 2\int_0^T(1+C_\bfb\norm{\eta(t)}{\clC^{1}(\overline\Omega)})\,\rmd t\notag\\
\mbox{and}\quad  &\xi_0\coloneqq\int_0^T C_\bfb^2\norm{\eta(t)}{L^2}^2\norm{\eta(t)}{\clC^{1}(\overline\Omega)}^2+\norm{P_{\clE_M^\top}^{\widehat\clU_M}f(t)}{H}^2\,\rmd t.\notag
\end{align}
Next we estimate~$\xi_1$ and~$\xi_0$. Denoting, for simplicity
\begin{equation}
w_0^M\coloneqq P_{\clE_M}^{\clE_M^\perp}w_0,\quad w_\tta^M\coloneqq P_{\clE_M}^{\clE_M^\perp}w_\tta,\quad \dnorm{P_M}{}\coloneqq\norm{P_{\widehat\clU_M}^{\clE_M^\top}}{\clL(H)},\quad\dnorm{P_M^\perp}{}\coloneqq\norm{P_{\clE_M^\top}^{\widehat\clU_M}}{\clL(H)},\notag
\end{equation}
and observing that~$P_{\widehat\clU_M}^{\clE_M^\top}=P_{\widehat\clU_M}^{\clE_M^\top}P_{\clE_M}^{\clE_M^\top}$ and~$P_{\clE_M^\top}^{\widehat\clU_M}=P_{\clE_M^\top}^{\clE_M}+P_{\clE_M^\top}^{\widehat\clU_M}P_{\clE_M}^{\clE_M^\top}$ we find
\begin{align}
\norm{\eta(t)}{H}&\le\dnorm{P_M}{}\left(\rme^{-\mu t}\norm{y_0^M}{H}+\rme^{-\mu (T- t)}\norm{y_\tta^M}{H}\right),\notag\\
\norm{\eta(t)}{\clC^{1}(\overline\Omega)}&\le\dnorm{P_M}{}\left(\rme^{-\mu t}\norm{y_0^M}{\clC^{1}(\overline\Omega)}+\rme^{-\mu (T- t)}\norm{y_\tta^M}{\clC^{1}(\overline\Omega)}\right),\notag\\
\norm{P_{\clE_M^\top}^{\widehat\clU_M}f(t)}{H}&\le\norm{P_{\clE_M^\top}^{\clE_M}f(t)}{H}+\norm{P_{\clE_M^\top}^{\widehat\clU_M}P_{\clE_M}^{\clE_M^\perp}f(t)}{H},\notag
\end{align}
which leads to
\begin{align}
\int_0^T \norm{\eta(t)}{\clC^{1}(\overline\Omega)}\,\rmd t&\le \dnorm{P_M}{}\mu^{-1}\left(\norm{y_0^M}{\clC^{1}(\overline\Omega)}+\norm{y_\tta^M}{\clC^{1}(\overline\Omega)}\right),\notag\\
\int_0^T \norm{\eta(t)}{\clC^{1}(\overline\Omega)}^2\,\rmd t&\le  \dnorm{P_M}{}^2(2\mu)^{-1}\left(2\norm{y_0^M}{\clC^{1}(\overline\Omega)}^2+2\norm{y_\tta^M}{\clC^{1}(\overline\Omega)}^2\right),\notag\\
\int_0^T \norm{\eta(t)}{L^2}^2\norm{\eta(t)}{\clC^{1}(\overline\Omega)}^2\,\rmd t&\le \norm{\eta}{L^\infty((0,T),H)}^2\dnorm{P_M}{}^2\mu^{-1}\left(\norm{y_0^M}{\clC^{1}(\overline\Omega)}^2+\norm{y_\tta^M}{\clC^{1}(\overline\Omega)}^2\right)\notag\\
&\le 2\dnorm{P_M}{}^4\mu^{-1}\left(\norm{y_0^M}{H}^2+\norm{y_\tta^M}{H}^2\right)\left(\norm{y_0^M}{\clC^{1}(\overline\Omega)}^2+\norm{y_\tta^M}{\clC^{1}(\overline\Omega)}^2\right).\notag
\end{align}
Hence,
\begin{subequations}\label{estchi01}
\begin{align}
\xi_1&\le 2T+2C_\bfb\dnorm{P_M}{}\mu^{-1}(\norm{y_0^M}{\clC^{1}(\overline\Omega)}+\norm{y_\tta^M}{\clC^{1}(\overline\Omega)}),\\
\xi_0&\le 2C_\bfb^2\dnorm{P_M}{}^4\mu^{-1}\left(\norm{y_0^M}{H}^2+\norm{y_\tta^M}{H}^2\right)\left(\norm{y_0^M}{\clC^{1}(\overline\Omega)}^2+\norm{y_\tta^M}{\clC^{1}(\overline\Omega)}^2\right)\notag\\
&\quad+2\int_0^T \norm{P_{\clE_M^\top}^{\clE_M}f(t)}{H}^2+ \norm{P_{\clE_M^\top}^{\widehat\clU_M}P_{\clE_M}^{\clE_M^\top}f(t)}{H}^2\,\rmd t.
\end{align}
\end{subequations}

By~\eqref{choiceM} it follows that
\begin{subequations}\label{choiceUMlam}
\begin{equation}
2\int_0^T \norm{P_{\clE_M^\top}^{\clE_M}f(t)}{H}^2\,\rmd t\le \frac{2\varepsilon^2}{20}\rme^{-3T}.\end{equation}
The next step concerns the choice of an appropriate~$\widehat\clU_M\coloneqq\clU_M$.
For this purpose, we use Lemma~\ref{L:obliproj} with~$\varrho=(\frac{\varepsilon^2}{40}\rme^{-3T})^\frac12\left(1+\int_0^T\norm{f(t)}{H}^2\,\rmd t\right)^{-\frac12}$, to choose~${\overline\jmath}\in\bbN$ large enough  and a set of~$M$ actuators~$U_M\subset\clG^{\overline\jmath}$ so that, with $\clU_M\coloneqq\linspan U_M$,
\begin{equation}
2\int_0^T\norm{P_{\clE_M^\top}^{\clU_M}P_{\clE_M}^{\clE_M^\top}f(t)}{H}^2\,\rmd t\le2\varrho^2 \int_0^T\norm{f(t)}{H}^2\,\rmd t\le \frac{\varepsilon^2}{20}\rme^{-3T}.
\end{equation}
Next, we can choose~$\mu>0$ large enough  so that
\begin{align}
2C_\bfb\dnorm{P_M}{}\mu^{-1}(\norm{y_0^M}{\clC^{1}(\overline\Omega)}+\norm{y_\tta^M}{\clC^{1}(\overline\Omega)})&\le T,\\
2C_\bfb^2\dnorm{P_M}{}^4\mu^{-1}\left(\norm{y_0^M}{H}^2+\norm{y_\tta^M}{H}^2\right)\left(\norm{y_0^M}{\clC^{1}(\overline\Omega)}^2+\norm{y_\tta^M}{\clC^{1}(\overline\Omega)}^2\right)&\le \frac{\varepsilon^2}{20}\rme^{-3T}.
\end{align}
\end{subequations}

From~\eqref{estchi01} and~\eqref{choiceUMlam}, it follows that
$\xi_1\le 3T$ and~$\xi_0\le  \frac{4\varepsilon^2}{20}\rme^{-3T}$.
Then, by~\eqref{bddQ} and~\eqref{choiceM}, using~$z(0)=P_{\clE_M^\top}^{\widehat\clU_M}y_0$, we arrive at
\begin{align}
&\norm{z}{L^\infty((0,T),H)}^2\le\rme^{3T}\left(\frac{\varepsilon^2}{20}\rme^{-3T}+\frac{4\varepsilon^2}{20}\rme^{-3T}\right)=\frac{\varepsilon^2}{4}.\notag
\end{align}

Finally, recalling~\eqref{control-cM} and since~$f$ is smooth in the time variable, it follows that
the control input~$t\mapsto u(t)=(U_M^\diamond)^{-1}c(t)\in\bbR^M$ is smooth as well.
Therefore,  by choosing a basis~$\bbF_{m_{\overline\jmath}}$ for~$\clU+\clG^{\overline\jmath}$ including the basis~$U_M$ for~$\clU_M\subseteq\clG^{\overline\jmath}$, we can conclude that
\begin{equation}\label{AC-ok-density}
\begin{split}
&\mbox{there exists~$\overline\jmath\in\bbN$, $\;\bbF_{m_{\overline\jmath}}\subset\clU+\clG^{\overline\jmath}$, and~$u\in \clC^{\infty}([0,T],\bbR^{m_{\overline\jmath}})$ such that}\\
&\hspace{2em}\linspan\bbF_{m_{\overline\jmath}}=\clU+\clG^{\overline\jmath},\quad  m_{\overline\jmath}\coloneqq\dim(\clU+\clG^{\overline\jmath}),\quad\mbox{and}\\
&\hspace{2em}\norm{\fkY(y_0;0,f+\bbF_{m_{\overline\jmath}}^\diamond u)(T)-y_\tta}{H}\le \frac\varepsilon2.
\end{split}
\end{equation}
In particular, since the tuple in~\eqref{fixedtuple} has been arbitrarily fixed, we can conclude that the Euler system~\eqref{evol-sys-bfc} is $(T,\fkU)$-approximately controllable with~$\fkU=\clC^\infty([0,T],\clU+\clG^{\overline\jmath})$.

\subsection{$(T,\fkU)$-approximately controllability with~$\fkU=\clC^\infty([0,T],\clU+\clG^{j-1})$ and~$j\ge1$.} \label{sS:proofACimitation}
In this section, we use the recursion~\eqref{seqclG} to replace the actuations in~$\clG^{j}\setminus\clG^{j-1}$ by actuations in~$\clU+\clG^{j-1}$, still driving the system to a neighborhood of the aim state~$w_\tta$, at time~$T$.

Let~$\varepsilon>0$ and~$\overline\jmath$ be as in~\eqref{AC-ok-density}. Assume that~$\overline\jmath\ge j\ge1$ and that
\begin{equation}\label{AC-Hypo-imitation}
\stepcounter{equation}\tag{hypo:\theequation}
\begin{split}
&\linspan\bbF_{m_{j}}=\clU+\clG^{j},\quad u^j\in \clC^\infty([0,T],\bbR^{m_j}),\quad m_j\coloneqq\dim(\clU+\clG^{j}),\\
&\norm{\fkY(w_0;0,f+\bbF_{m_{j}}^\diamond u^j)(T)-y_\tta}{H}\le \frac{2\overline\jmath-j}{2\overline\jmath}\varepsilon.
\end{split}
\end{equation}
Thus~$\bbF_{m_{j}}=\{\Phi_i\mid 1\le i\le m_j\}$ is a basis for~$\clU+\clG^j$ and we can write~$\bbF_{m_{j}}^\diamond u^j\coloneqq {\textstyle\sum\limits_{i=1}^{m_j}} u^j_i\Phi_i$.
In this section we show an analogue of~\eqref{AC-Hypo-imitation} also holds, with~$\bbF_{m_{j}}^\diamond u^j$ replaced by a control force taking values in~$\clU+\clG^{j-1}\subset \clU+\clG^{j}\subset\clC^{2,1}(\overline\Omega)$.

Recalling~\eqref{seqclG} and Lemma~\ref{L:findimGj} we can assume that
\begin{align}
&\bbF_{m_{j}}=\bbF_*{\;\textstyle\bigcup\;}\{\Phi_{m_j-i+1}=\clB(\psi_i)\phi_i\mid 1\le i\le n\}\label{bbF-osc}
\intertext{where}
&\psi_i\in\clG^0\subset\clC^{2,1}(\overline\Omega),\qquad \phi_i\in\clG^{j-1}\subset \clC^{2,1}(\overline\Omega),\notag\\
&\bbF_*\coloneqq\{\Phi_i\mid 1\le i\le m_j-n\},\quad\linspan\bbF_*=\clU+\clG^{j-1},\quad\clB(\psi_i)\phi_i\in\clG^{j}\setminus\clG^{j-1}\subset\clC^{2,1}(\overline\Omega).\notag
\end{align}
That is, we can write
\begin{align}
&\bbF_{m_j}^\diamond u^j(t)=\varkappa_0(t)+\overline\varrho_n(t)\in \clC^\infty([0,T],\clU+\clG^{j-1})+\clC^\infty([0,T],\clG^{j}_{j-1})\label{Fcontrol-given}
\intertext{with~$\clG^{j}_{j-1}\coloneqq\linspan(\clG^{j}\setminus\clG^{j-1})$ and}
&\varkappa_0(t)\in\clU+\clG^{j-1},\qquad\overline\varrho_n(t)={\textstyle\sum\limits_{i=1}^n} v_i(t)\clB(\psi_i)\phi_i,\qquad n\coloneqq\dim\clG^{j}_{j-1}<{m_j},\notag
\intertext{where}
&(\psi_i,\phi_i)\in \clG^{0}\times\clG^{j-1}\quad\mbox{and}\quad\clB(\psi_i)\phi_i\in\clG^{j}_{j-1}.\notag
\intertext{We can also assume that}
&\norm{\psi_i}{H}=\norm{\phi_i}{H}=1\quad\mbox{for all}\quad 1\le i\le n.\notag
\end{align}

We are looking for a control force in~$\clC^\infty([0,T],\clU+\clG^{j-1})$. We start by imitating (replacing) the component~$v_n(t)\clB(\psi_n)\phi_n$ with an appropriate control force taking values in~$\clU+\clG^{j-1}$. Of course, if~$v_n(t)=0$ we can just take the zero control force, however to simplify the exposition we treat both cases~$v_n(t)\ne0$ and~$v_n(t)=0$ analogously as follows.

We start by writing the given control force~\eqref{Fcontrol-given} as
\begin{equation}\label{Fcontrol}
\begin{split}
&\bbF_{m_j}^\diamond u^j(t)=\varkappa_0(t)+\overline\varrho_{n-1}(t)+\varrho_{n}(t)
\\
&\mbox{with}\quad\varrho_{n}(t)\coloneqq v_n(t)\clB(\psi_n)\phi_n\quad\mbox{and}\quad\overline\varrho_{n-1}(t)\coloneqq\overline\varrho_n-\varrho_n.
\end{split}
\end{equation}

We shall show that fast oscillating controls as
\begin{equation}\label{osccontrol}
\begin{split}
&\fkF_{\beta,K}(u^j(t))\coloneqq\varkappa_0(t)+\overline\varrho_{n-1}(t)+\beta^{-2}\varkappa_1(t) -\dot\varkappa_2(t),\quad\mbox{with}\\
&\varkappa_1(t)\coloneqq v_n^2(t)B(\psi_n,\psi_n)\mbox{ and }
\varkappa_2(t)\coloneqq2^\frac12\sin(\tfrac{\pi K t}{T})(\beta^{-1} v_n(t)\psi_n-\beta \phi_n),
\end{split}
\end{equation}
also steer the state to a neighborhood of the aim state, at time~$T$, provided~$\beta>0$ is small enough, and~$K=K(\beta)\in\bbN$ is large enough. Note that, recalling~\eqref{U}, we have that~$B(\psi_n,\psi_n)\in\clU$, hence~$\beta^{-2}\varkappa_1(t) -\dot\varkappa_2(t)\in\clU+\clG^{j-1}$.

We consider also an auxiliary perturbation of~$\bbF_{m_{j}}^\diamond u^j(t)$ as follows, with small~$\beta>0$,
\begin{equation}\label{undFcontrol}
\fkH_\beta(u^j)(t)\coloneqq \bbF_{m_{j}}^\diamond u^j(t)-\beta^2 B(\phi_n,\phi_n).
\end{equation}

\subsubsection{Comparing~$\bbF_{m_{j}}^\diamond u^j(t)$ with~$\fkH_\beta(u^j)$, for small~$\beta>0$.}
We investigate  the difference~$z\coloneqq \underline y-y$, with~$\underline y\coloneqq\fkY(y_0;0,f+\fkH_\beta(u^j))$ and~$y\coloneqq\fkY(y_0;0,f+\bbF_{m_{j}}^\diamond u)$, which solves
\begin{align}\notag
\dot z +B(z,z) +B(z,y) +B(y,z) =-\beta^2 B(\phi_n,\phi_n),\qquad z(0)=0,
\end{align}
see~\eqref{Fcontrol} and~\eqref{undFcontrol}. Then, using~\eqref{asymTrilin}  and~\eqref{estTrilin-h.C1.h},
we find
\begin{align}
\tfrac{\rmd}{\rmd t}\norm{z}{H}^2 &=-2\bfb(z,y,z)-2\beta^2 \bfb(\phi_n,\phi_n,z)\le2C_\bfb\norm{y}{\clC^1(\overline\Omega)}\norm{z}{H}^2+2C_\bfb\beta^2 \norm{\phi_n}{H}\norm{\phi_n}{\clC^1(\overline\Omega)}\norm{z}{H}\notag\\
&\le C_\bfb(2\norm{y}{\clC^1(\overline\Omega)}+1)\norm{z}{H}^2+C_\bfb\beta^4 \norm{\phi_n}{H}^2 \norm{\phi_n}{\clC^1(\overline\Omega)}^2\notag.
\end{align}

By Theorem~\ref{T:KatoShiri}, with~$k=2$, we have that
$y\in\clC([0,T],\clC^{1}(\overline\Omega))$. Thus, the Gronwall inequality gives
$
 \norm{z}{L^\infty((0,T),H)}^2\le \beta^{4}T\overline D_0\rme^{\overline D_1 T},
$
 with~$\overline D_1=C_\bfb(2\norm{y}{\clC^{0}([0,T],\clC^{1}(\overline\Omega))}+1)$ and~$\overline D_0=C_\bfb\norm{\phi_n}{L^\infty((0,T),H)}^2 \norm{\phi_n}{L^\infty((0,T),\clC^1(\overline\Omega))}^2$ independent of~$\beta$. 
Hence, we can choose~$\beta$ small enough  so that~$\norm{z(T)}{\bfH}\le\frac{1}{4n\overline\jmath}\varepsilon$, which combined with~\eqref{AC-ok-density} leads to
\begin{align}\label{AC-imitation-1}
\norm{\fkY(y_0;0,\bff+\fkH_\beta(u^j))(T)-y_\tta}{H}&\le \norm{z(T)}{H}+\norm{y(T)-y_\tta}{H}
\le\left(\tfrac{2\overline\jmath-j}{2\overline\jmath}+\tfrac{1}{4n\overline\jmath}\right)\varepsilon.
\end{align}

\subsubsection{Comparing~$\fkH_\beta(u^j)$ with~$\fkF_{\beta,K}(u^j)$, for fixed~$\beta>0$ and large~$K\in\bbN$.}\label{sS:imit-comp2}
Next, we investigate  the difference~$\overline z\coloneqq \overline y-\underline y$, with~$\underline y=\fkY(y_0;0,f+\fkH_\beta(u))$ and~$\overline y\coloneqq\fkY(y_0;0,f+\fkF_{\beta,K}(u))$.

We start by observing that, due to~$(y_0,0,f+\fkH(u))\in \bfX_2$ and Theorem~\ref{T:KatoShiri},  we have
 \begin{align}\label{bdd_underw}
 \{\underline y,\dot {\underline y},\p_{x_1}\underline y,\p_{x_2}\underline y\}\subset \clC([0,T]\times \clC^1(\overline\Omega)).
   \end{align}

By~\eqref{undFcontrol} and~\eqref{osccontrol}, the difference~$\overline z\coloneqq\overline y- \underline y$ solves
\begin{align}
\dot {\overline z}& =-B(\overline  z,\overline z)-B(\overline z,\underline y)-B(\underline y,\overline  z)+\beta^{-2}\varkappa_1 -\varrho_n+\beta^2 B(\phi_n,\phi_n) -\dot\varkappa_2.\notag
\end{align}
Introducing~$Z\coloneqq\overline  z+\varkappa_2$, we find
\begin{subequations}
\begin{align}
\dot Z =\dot {\overline z}+\dot\varkappa_2&=-B(Z-\varkappa_2,Z-\varkappa_2)-B(Z-\varkappa_2),\underline y)-B(\underline y,Z-\varkappa_2)\notag\\
&\quad+\beta^{-2}\varkappa_1 -\varrho_n+\beta^2 B(\phi_n,\phi_n),\label{dyn.Zimity1}\\
Z(0)&=z(0)=0,\qquad Z(T)=\overline z(T).\label{dyn.Zimity1-ifcond}
\end{align}
\end{subequations}

Observe that, recalling~\eqref{osccontrol} and~\eqref{Fcontrol},
\begin{align}
B(\varkappa_2(t),\varkappa_2(t))
&=2\sin^2(\tfrac{\pi K t}{T})\left(\beta^{-2} v_n^2(t)B(\psi_n,\psi_n)-v_n(t)\clB(\psi_n)\phi_n+\beta^{2}B(\phi_n,\phi_n)\right)\notag\\
&=2\sin^2(\tfrac{\pi K t}{T})\left(\beta^{-2}\varkappa_1(t) -\varrho_n(t)+\beta^{2}B(\phi_n,\phi_n)\right).\notag
\end{align}
Recalling also the relation~$2\sin^2(s)=1-\cos(2s)$, we find
\begin{align}
&\beta^{-2}\varkappa_1(t) -\varrho_n(t)+\beta^2B(\phi_n,\phi_n)-B(\varkappa_2(t),\varkappa_2(t))\notag\\
&\hspace{3em}=\cos(\tfrac{2\pi K t}{T})\left(\beta^{-2}\varkappa_1(t) -\varrho_n(t)+\beta^{2}B(\phi_n,\phi_n)\right)\eqqcolon h_{\beta,K}(t).\label{hbK}
\end{align}
Thus, from~\eqref{dyn.Zimity1}, we obtain
\begin{align}
\dot Z&=-B(Z,Z)+\clB(\varkappa_2)Z-\clB(\underline y)(Z-\varkappa_2)+h_{\beta,K}\notag\\
&=-B(Z,Z)+\clB(\varkappa_2-\underline y)Z+\clB(\underline y)\varkappa_2+h_{\beta,K}.\label{dyn.Zimity2}
\end{align}
Testing with~$2Z$ and using~\eqref{asymTrilin} and~\eqref{estTrilin-h.C1.h}, we find
\begin{align}
\tfrac{\rmd}{\rmd t}\norm{Z}{H}^2
&=2\bfb(Z,\varkappa_2-\underline y,Z)_{H}+2(\clB(\underline y)\varkappa_2+h_{\beta,K},Z)_{H}\notag\\
&\le2C_\bfb\norm{\varkappa_2-\underline y}{\clC^{1}}\norm{Z}{H}^2+2(\clB(\underline y)\varkappa_2+h_{\beta,K},Z)_{H}.
\notag
\end{align}
By~\eqref{bdd_underw} and~$\varkappa_2\in\clC^0([0,T],\clC^{2,1}(\overline\Omega))$ we have that
\begin{align}
\tfrac{\rmd}{\rmd t}\norm{Z}{H}^2&\le C_1\norm{Z}{L^2}^2+2(h,Z)_{L^2},
\qquad
\mbox{with}\quad h\coloneqq\clB(\underline y)\varkappa_2+h_{\beta,K},\label{dtZimit-p1}
\end{align}
for some constant~$C_1=C_1(\beta)>0$ independent of~$K$. By the Young inequality, we find
\begin{equation}
\tfrac{\rmd}{\rmd t}\norm{Z}{H}^2
\le (C_1+1)\norm{Z}{H}^2 +\norm{h}{H}^{2}\notag
\end{equation}
and, by the Gronwall inequality (in differential form~\cite[Ch.~3, Sect.~1.1.3]{Temam97}), we obtain
\begin{equation}
\norm{Z(t)}{H}^2\le\rme^{(C_1+1)T}\int_0^T\norm{h(s)}{H}^{2}\,\rmd s\le C_2,\qquad t\in[0,T],\label{Zimit-bdd-H}
\end{equation}
with~$C_2$ independent of~$K$; due to~\eqref{Fcontrol}, ~\eqref{osccontrol}, \eqref{hbK}, \eqref{bdd_underw}, and~\eqref{dtZimit-p1}.

However, note that~\eqref{Zimit-bdd-H} does not give us  yet that at final time~$t=T$, 
$\norm{Z(T)}{H}^2$ is small for large~$K$, as wanted. To this purpose, we need some further estimates. We start with an estimate for the vorticity function~$w\coloneqq (\curl Z$; from~\eqref{dyn.Zimity2} and~\eqref{dtZimit-p1}, we find
\begin{align}
\dot w&=-Z\cdot\nabla w+(\varkappa_2-\underline y)\cdot\nabla w+Z\cdot\nabla(\curl(\varkappa_2-\underline y)+(\curl h,\notag
\end{align}
which we test with~$2w$ to obtain
\begin{align}
\tfrac{\rmd}{\rmd t}\norm{w}{L^2}^2&=2( Z\cdot\nabla\nabla^\perp(\varkappa_2-\underline y),w)_{L^2}+2((\curl h,w)_{L^2}.\notag
\end{align}
By~\eqref{bdd_underw} and~$\varkappa_2\in\clC^0([0,T],\clC^{2,1}(\overline\Omega))$, we find~$\nabla\nabla^\perp(\varkappa_2-\underline y)(t)\in\clC^0(\overline\Omega)$, and by~\eqref{Zimit-bdd-H},
\begin{align}
\tfrac{\rmd}{\rmd t}\norm{w}{L^2}^2&\le 2\norm{\nabla\nabla^\perp(\varkappa_2-\underline y)}{\clC^{0}}\norm{Z}{H}\norm{w}{L^2}+2\norm{(\curl h}{L^2}\norm{w}{L^2}\notag\\
&\le 2C_3\norm{w}{L^2}+2\norm{\diver h}{L^2}\norm{w}{L^2}\le \norm{w}{L^2}^2+2(C_3^2+\norm{\diver h}{L^2}^2)\notag
\end{align}
with~$C_3$ independent of~$K$.  Then, the Gronwall inequality leads to
\begin{align}
\norm{w(t)}{L^2}^2&\le 2\rme^{T}\int_0^TC_3^2+\norm{\diver h(s)}{L^2}^2\,\rmd s\le C_4\label{w.Zimity-bdd-L2}
\end{align}
with~$C_4$ independent of~$K$.
Now, from~\eqref{Zimit-bdd-H}, \eqref{w.Zimity-bdd-L2}, and Lemma~\ref{L:vect-vort-est}, it follows
\begin{align}
\norm{Z(t)}{W^{1,2}(\Omega)}^2&\le  C_5\label{Zimity-bdd-H1}
\end{align}
with~$C_5$ independent of~$K$.

Next, by~\eqref{osccontrol} and the Gronwall inequality (\cite[Ch.~III, Sect.~1.1.3, Ineq.~(1.27)]{Temam97}) applied to~\eqref{dtZimit-p1}, we obtain
\begin{subequations}\label{ZoscAuxest-R2}
\begin{align}
\norm{Z(t)}{H}^2 &\le \int_0^t\sin(\tfrac{\pi K s}{T})\rme^{ C_1(t-s)} G_1(s)\,\rmd s+\int_0^t \cos(\tfrac{2\pi K s}{T})\rme^{ C_1(t-s)}G_2(s)\,\rmd s,\label{ZoscAuxest-R2int}\\
\quad\mbox{with}\hspace{2em}
G_1&\coloneqq 2^\frac32(\clB(\underline y)(\beta^{-1} v_n(t)\psi_n-\beta \phi_n),Z)_H,\\
G_2&\coloneqq2(\beta^{-2}\varkappa_1 -\varrho_n+\beta^{2}B(\phi_n,\phi_n),Z)_H.
\end{align}
\end{subequations}

Our next step is to we show that~$\{G_1,G_2\}\subset W^{1,1}(0,T)$. Denoting, for simplicity,
\begin{equation}\notag
\Psi_{n,\beta}\coloneqq\beta^{-1} v_n(t)\psi_n-\beta \phi_n,
\end{equation}
we observe that, using~\eqref{estTrilin-h.C1.h} and~\eqref{bdd_underw},
\begin{subequations}\label{Gs-in-W12-1}
\begin{align}
\norm{G_1}{\bbR}&\le2^\frac32C_\bfb\left(\norm{\Psi_{n,\beta}}{\clC^1(\overline\Omega)}\norm{\underline y}{H}+\norm{\Psi_{n,\beta}}{H}\norm{\underline y}{\clC^1(\overline\Omega)}\right)\norm{Z}{H}\le C_6\norm{Z}{H},\\
\norm{G_2}{\bbR}&\coloneqq2\left(\norm{\varrho_n}{H}+\beta^{-2}\norm{v_n}{L^\infty(0,T)}^2 C_\bfb\norm{\psi_n}{\clC^1(\overline\Omega)}\norm{\psi_n}{H}+\beta^{2}C_\bfb\norm{\phi_n}{\clC^1(\overline\Omega)}\norm{\phi_n}{H}\right)\norm{Z}{H}\notag\\
&\le C_7\norm{Z}{H},
\end{align}
\end{subequations}
with~$C_6$ and~$C_7$ independent of~$K\in\bbN_+$. Furthermore,
by~\eqref{ZoscAuxest-R2} and~\eqref{Gs-in-W12-1}, we find
\begin{align}
\norm{Z(t)}{H}^2 &\le C_8\int_0^t\norm{G_1(s)}{\bbR}+\norm{G_2(s)}{\bbR}\,\rmd s\le C_{9}\int_0^t 1+\norm{Z(s)}{H}^2\,\rmd s\notag\\
&\le C_{9}t+C_{9}\int_0^t \norm{Z(s)}{H}^2\,\rmd s\notag
\end{align}
and, by the  Gronwall inequality (in integral form~\cite[Eqs.~(5)--(6)]{Gronwall19}) it follows
\begin{align}\label{ZbddV'}
\norm{Z(t)}{H}^2 &\le  C_{9}t\rme^{C_{9}t}\le  C_{9}T\rme^{C_{9}T},\quad\mbox{for all }t\in(0,T).
\end{align}

By~\eqref{Gs-in-W12-1}, \eqref{ZbddV'}, and~\eqref{bdd_underw}, it follows that
\begin{align}\label{Gs-in-W12-1a}
\norm{G_1(t)}{\bbR}+\norm{G_2(t)}{\bbR}&\le C_{10},\quad\mbox{for all } t\in(0,T).
\end{align}

Next, for the time derivatives, we write
\begin{subequations}\label{dtG12}
\begin{align}
2^{-\frac32}\dot G_1&= \bfb(\dot\Psi_{n,\beta},\underline y,Z)+\bfb(\underline y,\dot\Psi_{n,\beta},Z)+\bfb(\Psi_{n,\beta},\dot{ \underline y},Z)\notag\\
&\quad+\bfb(\dot{\underline y},\Psi_{n,\beta},Z)+\bfb(\underline y,\Psi_{n,\beta}, \dot Z)+\bfb(\Psi_{n,\beta},\underline y, \dot Z),\\
2^{-1}\dot G_2&=(\beta^{-2}\dot\varkappa_1 -\dot\varrho_n,Z)_H+(\beta^{-2}\varkappa_1 -\varrho_n,\dot Z)_H+\beta^{2}\bfb(\phi_n,\phi_n,\dot Z).
\end{align}
\end{subequations}

Using~\eqref{estTrilin-h.C1.h}, \eqref{bdd_underw} and~\eqref{ZbddV'}, 
\begin{subequations}\label{dtG1-aux}
\begin{align}
&\norm{\bfb(\dot\Psi_{n,\beta},\underline y,Z)+\bfb(\underline y,\dot\Psi_{n,\beta},Z)}{\bbR}\notag\\
&\hspace{1em}\le C_\bfb\norm{\dot\Psi_{n,\beta}}{H}\norm{\underline y}{\clC^1(\overline\Omega)}\norm{Z}{H}+C_\bfb\norm{\underline y}{H}\norm{\dot\Psi_{n,\beta}}{\clC^1(\overline\Omega)}\norm{Z}{H}\le C_{11},\\
&\norm{\bfb(\Psi_{n,\beta},\dot{ \underline y},Z)+\bfb(\dot{ \underline y},\Psi_{n,\beta},Z)}{\bbR}\notag\\
&\hspace{1em}\le C_\bfb\norm{\Psi_{n,\beta}}{H}\norm{\dot{\underline y}}{\clC^1(\overline\Omega)}\norm{Z}{H} +C_\bfb\norm{\dot{\underline y}}{H}\norm{\Psi_{n,\beta}}{\clC^1(\overline\Omega)}\norm{Z}{H}
\le C_{11}
 \intertext{and, using~\eqref{estTrilin-C1.C2.V'},}
 &\norm{\bfb(\underline y,\Psi_{n,\beta},\dot Z)+\bfb(\Psi_{n,\beta},\underline y,\dot Z)}{\bbR}\notag\\
&\hspace{1em}\le C_\bfb\norm{(\underline y, \nabla\underline y)}{\clC^1(\overline\Omega)}\norm{(\Psi_{n,\beta},\nabla\Psi_{n,\beta})}{\clC^1(\overline\Omega)}\norm{\dot Z}{V'}\le C_{11}\norm{\dot Z}{V'},
\end{align}
with~$C_{11}$ independent of~$K\in\bbN_+$.
\end{subequations}
Next, by~\eqref{dyn.Zimity2} and~\eqref{estTrilin-l4.V.l4}, with~$h_{\beta,K}$ as in~\eqref{hbK},
\begin{align}
\norm{\dot Z}{V'}&\le \norm{B(Z,Z)}{V'}+\norm{\clB(\varkappa_2-\underline y)Z}{V'} +\norm{\clB(\underline y)\varkappa_2+h_{\beta,K}}{V'}\notag\\
&\le C_\bfb(\norm{Z}{L^4}+2\norm{\varkappa_2-\underline y}{L^4})\norm{Z}{L^4}+2C_\bfb\norm{\underline y}{L^4}\norm{\varkappa_2}{L^4}+\norm{h_{\beta,K}}{V'}\notag\\
&\le C_{12}(\norm{Z}{W^{1,2}}+2\norm{\varkappa_2-\underline y}{\clC^0})\norm{Z}{W^{1,2}}+ C_{12}\norm{\underline y}{\clC^0}\norm{\varkappa_2}{\clC^0}+\norm{h_{\beta,K}}{V'}\le C_{13},\label{dyn.Zimity-dt}
\end{align}
with~$C_{13}$ independent of~$K\in\bbN_+$, where we have used~\eqref{Zimity-bdd-H1} together with the Sobolev embedding~$W^{1,2}(\Omega)\xhookrightarrow[]{} L^4(\Omega)$ holding for bounded domains~$\Omega\subset\bbR^2$; see~\cite[Cor.~4.53]{DemengelDem12}.

By~\eqref{dtG12} and~\eqref{dtG1-aux} it follows that
\begin{subequations}\label{bdd-dtG}
\begin{equation}
\norm{\dot G_1(t)}{\bbR}\le C_{14},\quad\mbox{for all } t\in(0,T),
\end{equation}
with~$C_{14}$ independent of~$K\in\bbN_+$.
Finally, by~\eqref{dtG12}, \eqref{estTrilin-C1.C2.V'}, \eqref{ZbddV'}, and~\eqref{dyn.Zimity-dt}, we also have, with~$\zeta_{n,\beta}\coloneqq\beta^{-2}\varkappa_1 -\varrho_n$,
\begin{align}
\norm{\dot G_2}{\bbR}&\le2 \left( \norm{\dot\zeta_{n,\beta}}{H}\norm{Z}{H}+\norm{\zeta_{n,\beta}}{V}\norm{\dot Z}{V'}\right)+\beta^{2}\tfrac12C_\bfb\norm{(\phi_n,\nabla\phi_n)}{\clC^1}^2\norm{\dot Z}{V'}\le C_{15},
\end{align}
for all~$t\in(0,T)$, with~$C_{15}$ independent of~$K\in\bbN_+$.
\end{subequations}
By combining~\eqref{Gs-in-W12-1a} with~\eqref{bdd-dtG} we find
\begin{align}
\norm{\rme^{ C_3(t-\Bigcdot)}G_1}{W^{1,\infty}(0,T)}+\norm{\rme^{ C_3(t-\Bigcdot)}G_2}{W^{1,\infty}(0,T)}\le C_{16},\notag
\end{align}
 with a constant~$C_{16}$  independent of~$K\in\bbN_+$,
Hence, by using~\eqref{ZoscAuxest-R2} and Lemma~\ref{L:intsincosKTheta}, 
\begin{align}
\norm{Z(T)}{H}^2 &\le 2\tfrac{(\pi C+1)T}{\pi} C_{16}T K^{-1}\notag
\end{align}
with~$C=\norm{\Id}{\clL(W^{1,1}(0,T),L^\infty(0,T))}$. Therefore, now we can choose ~$K$ large enough so that $\norm{Z(T)}{H}\le\frac{1}{4n\overline\jmath}\varepsilon$, which combined with~\eqref{AC-imitation-1} and~\eqref{dyn.Zimity1-ifcond}, leads to the estimates
\begin{equation}\notag
\norm{\fkY(y_0;0,f+\fkF_{\beta,K}(u^j))(T)-y_\tta}{H}\le \norm{\overline z(T)}{H}+\norm{\underline y(T)-y_\tta}{H}\le\left(\tfrac{2\overline\jmath-j}{2\overline\jmath}+\tfrac{1}{2n\overline\jmath}\right)\varepsilon.
\end{equation}

\subsubsection{Iteration towards the control taking values in~$\clU+\clG^{j-1}$}
Observe that we can write the control force in~\eqref{osccontrol} as
\begin{equation}\notag
\fkF_{\beta,K}(u)(t)\eqqcolon \bbF_{{m_j}-1}^\diamond \overline u(t)={\textstyle\sum\limits_{j=1}^{{m_j}-1}} \overline u_j\Phi_j,
\end{equation}
where~$\overline u(t)\in\clC^\infty([0,T],\bbR^{{m_j}-1})$ and with the~$\Phi_j$ as in~\eqref{bbF-osc}.
 In other words we have an analogue of~\eqref{AC-Hypo-imitation} as
\begin{equation}\notag
\overline u\in\clC^\infty([0,T],\bbR^{{m_j}-1})\quad\mbox{and}\quad
\norm{\fkY(y_0;0,f+ \bbF_{{m_j}-1}^\diamond \overline u)(T)-y_\tta}{H}\le \left(\tfrac{2\overline\jmath-j}{2\overline\jmath}+\tfrac{1}{2n\overline\jmath}\right)\varepsilon.
\end{equation}

Repeating the arguments,~$n-1$ times more, where~$n=\dim\clG^{j}_{j-1}$, we arrive at a control input giving us the analogue of~\eqref{AC-Hypo-imitation} as follows, 
\begin{equation}\label{AC-Thes-imitation}
\stepcounter{equation}\tag{thes:\theequation}
\begin{split}
&\hspace{-.5em}\linspan\bbF_{m_{j-1}}=\clU+\clG^{j-1},\quad u^{j-1}\in\clC^\infty([0,T],\bbR^{m_{j-1}}),\quad m_{j-1}=\dim(\clU+\clG^{j-1}),\\
&\hspace{-.5em}\norm{\fkY(y_0;0,f+\bbF_{m_{j-1}}^\diamond u^{j-1})(T)-y_\tta}{H}\le \left(\tfrac{2\overline\jmath-j}{2\overline\jmath}+\tfrac{n}{2n\overline\jmath}\right)\varepsilon=\tfrac{2\overline\jmath-(j-1)}{2\overline\jmath}\varepsilon.
\end{split}
\end{equation}
In particular, since~$(\overline\jmath,j)$ is fixed and the tuple in~\eqref{fixedtuple} has been arbitrarily fixed, we can conclude that the Euler system~\eqref{evol-sys-bfc} is $(T,\fkU)$-approximately controllability with~$\fkU=\clC^\infty([0,T],\clU+\clG^{j-1})$.
In other words, we have shown that we have the implication
\begin{equation}\label{AC-ok-imitation}
\mbox{\eqref{AC-Hypo-imitation}\quad$\Longrightarrow$\quad \eqref{AC-Thes-imitation}},\qquad\mbox{where}\quad\overline\jmath\ge j\ge1.
\end{equation}

\subsection{$(T,\fkU)$-approximately controllability with~$\fkU=\clC^\infty([0,T],\clU)$}\label{sS:proofACiteration}
Here, we complete the proof of Theorem~\ref{T:approxcontrol} by using~\eqref{AC-ok-density} and by iterating the implication in~\eqref{AC-ok-imitation}.

Let~$\varepsilon>0$.
From~\eqref{AC-ok-density} we can find~${\overline\jmath}\in\bbN$ large enough and a control input~$u^{\overline\jmath}\in\clC^{\infty}([0,T],\bbR^{m_{\overline\jmath}})$ so that
\begin{equation}\notag
\begin{split}
&\norm{\fkY(y_0;0,f+\bbF_{m_{\overline\jmath}}^\diamond u^{\overline\jmath})(T)-y_\tta}{H}\le\tfrac12\varepsilon,\quad\mbox{with}\quad\bbF_{m_{\overline\jmath}}^\diamond u^{\overline\jmath}(t)\in
\clU+\clG^{\overline\jmath},
\end{split}
\end{equation}
where~$m_{\overline\jmath}=\dim(\clU+\clG^{\overline\jmath})$; see~\eqref{AC-ok-imitation}. If~${\overline\jmath}=0$ the proof of Theorem~\ref{T:approxcontrol} is finished. If~${\overline\jmath}\ge1$ we observe that~\eqref{AC-Hypo-imitation} holds with~$j=\overline\jmath$, and we use the implication~\eqref{AC-ok-imitation} to conclude the existence of a control input~$ u^{\overline\jmath-1}\in\clC^{\infty}([0,T],\bbR^{m_{\overline\jmath-1}})$ so that
\begin{equation}\notag
\begin{split}
&\norm{\fkY(y_0;0,f+\bbF_{m_{\overline\jmath-1}}^\diamond u^{\overline\jmath-1})(T)-y_\tta}{H}\le\tfrac{2\overline\jmath-(\overline\jmath-1)}{2\overline\jmath}\varepsilon,\quad\mbox{with}\quad\bbF_{m_{\overline\jmath-1}}^\diamond u^{\overline\jmath-1}(t)\in
\clU+\clG^{\overline\jmath-1}.
\end{split}
\end{equation}
Using the same implication $\overline\jmath-1$ times more, we arrive at an input~$u^0\in\clC^{\infty}([0,T],\bbR^{m_{0}})$, where~$m_0=\dim(\clU+\clG^{0})$,  so that
\begin{equation}\notag
\begin{split}
&\norm{\fkY(y_0;0,f+\bbF_{m_0}^\diamond u^{0})(T)-y_\tta}{H}\le\tfrac{2\overline\jmath-0}{2\overline\jmath}\varepsilon=\varepsilon,\quad\mbox{with}\quad \bbF_{m_0}^\diamond u^{0}(t)\in
\clU+\clG^{0}=\clU.
\end{split}
\end{equation}
Recall that~$\clG^{0}=\linspan G_0$ and~$U=G_0\bigcup \{\bfB(h,h)\mid h\in G_0\}$, thus~$\clU+\clG^{0}=\clU=\linspan U$. Then, with~$U= \bbF_{m_0}$ and~$u= u^{0}\in\clC^{\infty}([0,T],\bbR^{m_{0}})$, we can write
\begin{equation}\notag
\begin{split}
&\norm{\fkY(y_0;0,f+U^\diamond u)(T)-y_\tta}{H}\le\varepsilon,\quad\mbox{with}\quad U^\diamond u(t)\in
\clU.
\end{split}
\end{equation}

Therefore, since the tuple in~\eqref{fixedtuple} has been arbitrarily fixed, we can conclude that
the Euler system~\eqref{evol-sys-bfc} is $(T,\fkU)$-approximately controllability with~$\fkU=\clC^\infty([0,T],\clU)$. This finishes the proof of Theorem~\ref{T:approxcontrol}.
\qed

\section{Proof of Theorem~\ref{T:saturating}: reduction to vorticity formulation}\label{S:proofT:saturating-vort-form}

We start by noticing that the unit disk~$\Omega=\clD$ in Theorem~\ref{T:saturating} is simply connected. In this case the $\curl$ mapping in Lemma~\ref{L:vect-vort-est} is a bijection from~$H\bigcap W^{1,2}(\Omega)$ onto~$L^2(\Omega)$ and, by this reason, it is equivalent and convenient to work with the vorticity formulation. We start by looking at the dynamics of the vorticity function~$w\coloneqq \curl y$ of the vector field~$y$ satisfying~\eqref{evol-sys-U}, and by writing Theorem~\ref{T:saturating} in vorticity formulation.
 
%
\subsection{Solenoidal velocity fields and their vorticity}
Recalling the space~$H$ of solenoidal vector fields~\eqref{H}, and its subspace~$V=H\bigcap W^{1,2}(\Omega)^2$ in~\eqref{V}, with scalar product as in~\eqref{scalarprodV},
we define the Leray projection~$\Pi$ in~$(W^{1,2})^2=W^{1,2}(\Omega)^2$ as
\begin{subequations}\label{LerayPi}
\begin{align}
&\Pi\in\clL((W^{1,2})^2,V),\qquad\Pi z\coloneqq z+\nabla\phi_z,\\
\intertext{with~$\phi_z$ defined by}
&\Delta\phi_z=-\diver z,\qquad\bfn\cdot\nabla\phi_z\rest{\p\clD}=-\bfn\cdot z\rest{\p\clD}.
\end{align}
We can extend this projection to~$(W^{1,2}(\Omega)^2)'$ as
\begin{align}
&\Pi^\rme\in\clL(((W^{1,2})^2)',V'),&&\langle\Pi^\rme w,v\rangle_{V',V}\coloneqq\langle w, v\rangle_{((W^{1,2})^2)',(W^{1,2})^2}.
\end{align}
\end{subequations}

The latter is an extension of the former because, for~$w\in (W^{1,2})^2$, 
\begin{align}
\langle\Pi^\rme w,v\rangle_{V',V}=(w, v)_{(L^2)^2}=(\Pi w-\nabla\phi_w, v)_{(L^2)^2}=(\Pi w, v)_{H}=\langle\Pi w,v\rangle_{V',V}.\notag
\end{align}

Hereafter we shall denote the extension~$\Pi^\rme$ by~$\Pi$. Further, from~\cite[Ch.~1, Thm.~1.4]{Temam01} we have a ``characterization'' of the orthogonal complement~$H^\perp$ to~$H$ in~$(L^2)^2$ as~$H^\perp=\nabla W^{1,2}$. This means that the relation~$z=P_H z-\nabla\psi_z$ also holds, for~$z\in(L^2)^2$ with~$P_H z\in H$ and~$\psi_z\in W^{1,2}$. That is,  $P_H$ is the orthogonal projection in~$(L^2)^2$ onto~$H$ and extends of~$\Pi$. Note also that~$\Pi^\rme$ is an extension of~$P_H$ because, for~$w\in H\times H$
\begin{align}
\langle\Pi^\rme w,v\rangle_{V',V}=(w, v)_{(L^2)^2}=(P_H  w-\nabla\psi_w, v)_{(L^2)^2}=(P_H w, v)_{H}=\langle P_H  w,v\rangle_{V',V}.\notag
\end{align}

Next, we consider the adjoint of the $\curl$ operator in~\eqref{curl},
\begin{equation}\label{curl*}
 \curl^*\colon L^2\mapsto V',\qquad \langle\curl^* g, v\rangle_{V',V}\coloneqq (g,\curl v)_{L^2},
\end{equation}
and note that
 \begin{equation}\label{curl*form}
\curl^* g=(-\tfrac{\p g}{\p x_2}, \tfrac{\p g}{\p x_1}),\quad \mbox{for}\quad g\in W^{1,2}_0,
\end{equation}
where
\begin{equation}\label{W10}
W^{1,2}_0\coloneqq\{h\in W^{1,2}\mid h\rest{\p\clD}=0\}\subset W^{1,2}.
\end{equation}
We further introduce the (negative) Dirichlet Laplacian
\begin{equation}\label{dirLap}
-\Delta_\rmd\colon W^{1,2}_0\to W^{-1,2},\qquad \langle -\Delta_\rmd h, g\rangle_{W^{-1,2},W^{1,2}_0}\coloneqq(\curl^* h,\curl^*g )_{H}.
\end{equation}
with~$W^{-1,2}\coloneqq(W^{1,2}_0)'$. 
Note that~$(\curl^* h,\curl^*g )_{H}=(\nabla h,\nabla g )_{(L^2)^2}$. The domain is
\begin{equation}\label{DA}
\rmD(-\Delta_\rmd)\coloneqq\{h\in L^2\mid -\Delta_\rmd y\in L^2\}= W^{1,2}_0{\,\textstyle\bigcap\,} W^{2,2}.
\end{equation}

For each vector field~$y\in V$, we define the stream function~$\phi^y\in \rmD(-\Delta_\rmd)$ as the solution of
$-\Delta_\rmd\phi^y=\curl y$, that is,
\begin{equation}\label{stream_fun}
\phi^y=(-\Delta_\rmd)^{-1}\curl y.
\end{equation}

 Further, we observe that
\begin{equation}\notag
\curl\curl^*\phi^y=-\Delta_\rmd \phi^y=\curl y,\qquad \diver\curl^*\phi^y=0,\qquad \bfn\cdot\curl^*\phi^y\rest{\p\clD}=0.
\end{equation}
Hence, for the difference~$z\coloneqq y-\curl^*\phi^y$ we find the relations
\begin{equation}\notag
\curl z=0,\qquad \diver z=0,\qquad \bfn\cdot z\rest{\p\clD}=0.
\end{equation}
Since~$\clD$ is simply connected we find~$z=0$ (cf.~\cite[App.~I, Sect.~1, Prop.1.4 and Lem.~1.6]{Temam01}, \cite[Thm.~1]{AuchAlex01}).
Hence, we obtain~$y=\curl^*\phi^y=\curl^*(-\Delta_\rmd)^{-1}\curl y$ and 
\begin{align}\notag
 \curl B(y,y)=\curl(\langle y\cdot\nabla\rangle y) =y\cdot\nabla \curl y=(\curl^*(-\Delta_\rmd)^{-1}\curl y)\cdot\nabla \curl y
\end{align}
(cf.~\cite[Eq.~(2.23)]{ClopeauMikRob98},~\cite[Ch.~I, Sect.~11.A]{ArnKhesin98}), which we can write as
\begin{align}\label{curl-nonl}
 \curl B(y,y)=\curl(\langle y\cdot\nabla\rangle y) =\{(-\Delta_\rmd)^{-1}\curl y, \curl y\}_\fkp,
\end{align}
with the bilinear Poisson bracket on~$W^{1,2}$ defined as
\begin{equation}\label{poisson_b}
\begin{split}
&\{\Bigcdot , \Bigcdot\}_\fkp\colon W^{1,2}\times W^{1,2}\to L^1,\\
&\{h, g\}_\fkp\coloneqq(\curl^*h,\nabla g)_{\bbR^2}=-\frac{\p h}{\p x_2}\frac{\p g}{\p x_1}+\frac{\p h}{\p x_1}\frac{\p g}{\p x_2}.
\end{split}
\end{equation}

From~\eqref{evol-sys-U} and~\eqref{curl-nonl}, we find the evolutionary equation
\begin{subequations}\label{vort-sys}
\begin{align}
&\dot{w}+ \{(-\Delta_\rmd)^{-1}w,w\}_\fkp =f^\rmc+
B^\rmc u,\qquad w(0)=w_0,
\end{align}
 for the vorticity~$w=\curl y$, with~$w_0\coloneqq\curl u_0$ and 
\begin{align}
f^\rmc\coloneqq\curl f,\qquad B^\rmc\coloneqq\curl\circ U^\diamond,\quad B^\rmc u\coloneqq \sum_{k=1}^M z_k\phi_k,\quad\phi_k\coloneqq\curl F_k,\quad F_k\in U.
\end{align}
\end{subequations}

We want to write the analogue of Definition~\ref{D:saturVlin} in terms of vorticity functions. For this purpose we start by introducing the operator
  \begin{align}\label{Buz-vort}
\fkB(z)y&\coloneqq\{(-\Delta_\rmd)^{-1}h,g\}_\fkp-\{h,(-\Delta_\rmd)^{-1}g\}_\fkp
\end{align}
and, for subspaces~$Z\subseteq W^{-1,2}$,  $Y\subseteq W^{-1,2}$, and~$\clX\subseteq W^{-1,2}$, we also introduce the subspace
  \begin{align}\label{BZH-clH}
\fkB_\clX(Z)Y\coloneqq\linspan \{\clX{\,\textstyle\bigcap\,}\fkB(z)y\mid (z,y)\in Z\times Y \}.
\end{align}
\begin{definition}\label{D:saturVlin-vort}
Given a subspace~$\clX\subset W^{-1,2}$, a finite subset~$S_0\subset \clX$ is said $\clX$-saturating if the sequence~$(\clS^j)_{j\in\bbN}$ of subspaces~$\clS^j\subset \clX$ defined  by
\begin{align}\label{seqclS}
\clS^0&\coloneqq\linspan S_0,\qquad
\clS^{j+1}\coloneqq\clS^{j}+\fkB_\clX(\clS^0)\clS^{j},
\end{align}
is such that the union~${\textstyle\bigcup_{j\in\bbN}}\clS^j$ is dense in~$W^{-1,2}$. 
\end{definition}

 \begin{lemma}\label{L:saturVlin-vort}
 Let~$\Omega\subset\bbR^2$ be a bounded smooth simply-connected spatial domain. A set of vector fields~$G_0$ is $\clH$-saturating if, and only if,  the set of functions~$S_0\coloneqq \curl G_0$ is $(\curl \clH)$-saturating.
  \end{lemma}
The proof is given in the Appendix, Section~\ref{Apx:proofL:saturVlin-vort}.

\begin{theorem}\label{T:saturating-vort}
The set~$E$ in~\eqref{satset-vort} is~$\clC^{1,1}(\overline\clD)$-saturating.
\end{theorem}
The proof is given in Section~\ref{S:proofT:saturating-vort}.

\begin{lemma}\label{L:saturating-vort-clXs}
Let~$\clX_1$, $\clX_2$ be subspaces with continuous inclusions~$\clX_1\subseteq\clX_2\subset  W^{-1,2}$, a finite subset~$S_0\subset \clX_1$  is~$\clX_1$-saturating only if it is ~$\clX_2$-saturating.
\end{lemma}
\begin{proof}
It is sufficient to observe that the corresponding saturating sequences as in~\eqref{seqclS},
\[\clS_k^0=\clS^0\coloneqq\linspan S_0, \quad
\clS_k^{j+1}\coloneqq\clS_k^{j}+\fkB_{\clX_k}(\clS^0)\clS_k^{j}, \qquad k\in\{1,2\},
\] satisfy~$\clS_1^{j}\subseteq \clS_2^{j}$, which we can show by induction on~$j\in\bbN$.
\end{proof}

\subsection{Proof of Theorem~\ref{T:saturating}}
Straightforward by Theorem~\ref{T:saturating-vort} and Lemmas~\ref{L:saturVlin-vort} and~\ref{L:saturating-vort-clXs}.

\section{Proof of Theorem~\ref{T:saturating-vort}}\label{S:proofT:saturating-vort}

\subsection{Polar coordinates}\label{sS:auxRes}
It is convenient to consider polar coordinates $(r,\theta) \in [0,1) \times [0,2\pi)$
in the disc~$\clD$. 
From
\begin{equation}
x_1=r\cos(\theta),\qquad x_2=r\sin(\theta),\notag
\end{equation}
 relating Cartesian~$(x_1,x_2)$ and polar~$(r,\theta)$ coordinates, we can find that
 \begin{align} 
\tfrac{\p}{\p r}&=\tfrac{\p x_1}{\p r}\tfrac{\p}{\p x_1}+\tfrac{\p x_2}{\p r}\tfrac{\p}{\p x_2}
=\cos(\theta) \tfrac{\p}{\p x_1}+\sin(\theta)\tfrac{\p}{\p x_2},\notag\\
\tfrac{\p}{\p \theta}&=\tfrac{\p x_1}{\p \theta}\tfrac{\p}{\p x_1}+\tfrac{\p x_2}{\p \theta}\tfrac{\p}{\p x_2}
=-r\sin(\theta)\tfrac{\p}{\p x_1}+r\cos(\theta)\tfrac{\p}{\p x_2},\notag
\end{align}
 which we can rewrite as
\begin{align} 
\begin{bmatrix}\tfrac{\p}{\p r}\\\tfrac{\p}{\p \theta}\end{bmatrix}
&=\begin{bmatrix}\cos(\theta) & \sin(\theta)\\
-r\sin(\theta)& r\cos(\theta)\end{bmatrix}\begin{bmatrix}\tfrac{\p}{\p x_1}\\\tfrac{\p}{\p x_2} 
\end{bmatrix}\notag
\intertext{and, after inversion,}
\begin{bmatrix}\tfrac{\p}{\p x_1}\\\tfrac{\p}{\p x_2} 
\end{bmatrix}\notag
&=\frac1{r}
\begin{bmatrix}r\cos(\theta) & -\sin(\theta)\\
r\sin(\theta)& \cos(\theta)\end{bmatrix}
\begin{bmatrix}\tfrac{\p}{\p r}\\\tfrac{\p}{\p \theta}
\end{bmatrix}\notag
\end{align}
that is,
 \begin{equation}\label{grad-polar}
\tfrac{\p}{\p x_1}=\cos(\theta)\tfrac{\p }{\p r}-\tfrac1r\sin(\theta)\tfrac{\p}{\p  \theta},\qquad
\tfrac{\p}{\p x_2}=\sin(\theta)\tfrac{\p }{\p r}+\tfrac1r\cos(\theta)\tfrac{\p}{\p  \theta}.
\end{equation}

Now, for the Poisson bracket~\eqref{poisson_b}, by direct computations, we find
 \begin{align}
\{f, g\}_\fkp&=-\left(\sin(\theta)\tfrac{\p f}{\p r}+\tfrac1r\cos(\theta)\tfrac{\p f}{\p  \theta}\right)\left(\cos(\theta)\tfrac{\p g}{\p r}-\tfrac1r\sin(\theta)\tfrac{\p g}{\p  \theta}\right)\notag\\
&\quad+\left(\cos(\theta)\tfrac{\p f}{\p r}-\tfrac1r\sin(\theta)\tfrac{\p f}{\p  \theta}\right)\left(\sin(\theta)\tfrac{\p g}{\p r}+\tfrac1r\cos(\theta)\tfrac{\p g}{\p  \theta}\right)\notag,
  \end{align}
 which leads us to the expression
  \begin{align}\label{poisson_b-pol}
\{f, g\}_\fkp&=\frac1r\left(\frac{\p f}{\p  r}\frac{\p g}{\p  \theta}-\frac{\p f}{\p  \theta}\frac{\p g}{\p  r}\right).
  \end{align}
By direct computations, we also find  the known formal expression for the  Laplacian
\begin{equation}\label{A-pol}
\Delta=\frac{\partial^2} {\partial r^2}+\frac{1}{r}\frac{\partial }{\partial r}+\frac{1}{r^2}\frac{\partial^2 }{\partial \theta^2}.
  \end{equation}
  
 Since saturating sets are defined through the iteration~\eqref{seqclS}, we shall need suitable properties of the operator~$\fkB$ in~\eqref{Buz-vort}, involving Poisson brackets~$\{\Bigcdot,\Bigcdot\}_\fkp$ and~$(-\Delta_\rmd)^{-1}$. 
  
Recalling Definition~\ref{D:saturVlin-vort}, Lemma~\ref{L:saturVlin-vort}, and Theorem~\ref{T:saturating-vort}, we are looking for a finite set~$E\subset L^2$ such that its span~$\clS^0=\linspan E$, as used in the iterative process~\eqref{seqclS}, leads to the increasing sequence~$(\clS^j)_{j\in\bbN}$ as in Lemma~\ref{L:saturVlin-vort}. This process generates new functions through \begin{equation}\label{generator}
h\mapsto \fkB(g)h,\quad\mbox{with}\quad g\in\clS^0,
\end{equation}
thus we shall call the function~$g$ a generator function.

  \subsection{Dirichlet Laplacian for a class of functions}
 Here, we derive explicit expressions for~$(-\Delta_\rmd)^{-1}h$ and~$\{h,g\}_\fkp$,  for suitable functions~$h,g$.
We denote the functions
\begin{subequations}\label{cs_mn} 
\begin{align}
 \fkc_{(m,n)}=\fkc_{(m,n)}(r,\theta)&\coloneqq r^{2(m+n)}\cos(n\theta),&&\quad\mbox{for}\quad m\in\bbN,\quad n\in\bbN;\\
\fks_{(m,n)}=\fks_{(m,n)}(r,\theta)&\coloneqq r^{2(m+n)} \sin(n\theta),&&\quad\mbox{for}\quad m\in\bbN,\quad n\in\bbN_+;
  \end{align}
   \end{subequations} 
and the constant
\begin{equation}\label{xi_mn} 
  \xi_{(m,n)}\coloneqq4(m+n)^2-n^2,\quad\mbox{for}\quad m\in\bbN,\quad n\in\bbN.
  \end{equation}
 
Note that, from~\eqref{grad-polar} and~\eqref{cs_mn}, we can see that
 \begin{equation}\label{cs-in-W12} 
  \fkc_{(m,n)} \in\clC^{1,1}(\overline\clD)\quad\mbox{and}\quad  \fks_{(m,n)} \in\clC^{1,1}((\overline\clD),\quad\mbox{if~$m+n>0$}.
  \end{equation}

Recall that we distinguish between the formal Laplacian~$\Delta$ and the Dirichlet Laplacian~$\Delta_\rmd$. These operators coincide in~$\rmD(-\Delta_\rmd)$.
    \begin{lemma}\label{L:lap_cs}
    We have the  identities $\Delta \fkc_{(0,0)}=0$ and, for~$(m,n)\in\bbN\times\bbN$,
  \begin{align}
\Delta \fkc_{(m,n)}&=\xi_{(m,n)}\fkc_{(m-1,n)},&&\quad\mbox{if}\quad n+m\ge1,\notag\\
 \Delta \fks_{(m,n)}&=\xi_{(m,n)}\fks_{(m-1,n)},&&\quad\mbox{if}\quad n\ge1.\notag
  \end{align}   
 \end{lemma}   
  \begin{proof}
It is straightforward to see that~$\fkc_{(0,0)}=1$ is a constant function and that~$\Delta \fkc_{(0,0)}=0$, using~\eqref{A-pol}.
Next,  for ~$n+m\ge1$ we find, using again~\eqref{A-pol},
    \begin{align}
   \Delta \fkc_{(m,n)}&= \Bigl(2(n+m)(2(n+m)-1) +2(n+m)-n^2\Bigr) r^{2(n+m)-2}\cos(n\theta)\notag\\
  &=\xi_{(m,n)}\fkc_{(m-1,n)}.\notag
    \end{align}   
Similarly, we can find~$\Delta \fks_{(m,n)}=\xi_{(m,n)}\fks_{(m-1,n)}$.
   \end{proof}

     \begin{corollary}\label{C:invlap_cs}
    We have the  identities
  \begin{align}
  &(-\Delta_\rmd)^{-1}(\xi_{(m+1,0)}r^{2m})=(1-r^{2(m+1)})\in W^{1,2}_0,\qquad\mbox{for}\quad m\in\bbN.\notag
  \intertext{Further, in the case~$n\ge1$ we have that}
    &(-\Delta_\rmd)^{-1}( \xi_{(m+1,n)}\fkc_{(m,n)}-\xi_{(k+1,n)}\fkc_{(k,n)})= r^{2n}(r^{2(k+1)}-r^{2(m+1)})\cos(n\theta)\in W^{1,2}_0,\notag\\
 &(-\Delta_\rmd)^{-1}( \xi_{(m+1,n)}\fks_{(m,n)}-\xi_{(k+1,n)}\fks_{(k,n)})= r^{2n}(r^{2(k+1)}-r^{2(m+1)})\sin(n\theta)\in W^{1,2}_0,\notag    
  \end{align}   
  for all~$k\in\bbN$ and~$m\in\bbN$.
\end{corollary}   
\begin{proof}
Straightforward, from Lemma~\ref{L:lap_cs}. 
\end{proof}  
  
  \subsection{Poisson bracket. Involving the generator with vanishing periodic frequency.} 
We consider~\eqref{generator} with the generator~$1-r^2\in\clS^0=\linspan E$ independent of~$\theta$; see~\eqref{satset-vort}.
 \begin{lemma}\label{L:poiss_cs0km}
    We have the  identities
  \begin{subequations}
  \begin{align}
  &\fkB(1-r^2)\fkc_{(m,0)}=0,\quad\mbox{for all~$m\in\bbN$}.\notag
  \intertext{Further, for all~$n\in\bbN_+$ and all~$(m,k)\in\bbN\times\bbN$, }
 &\fkB(1-r^2) (\xi_{(m+1,n)}\fkc_{(m,n)}-\xi_{(k+1,n)}\fkc_{(k,n)})\notag\\
 &\hspace{3em}=+\left(\frac{n}4\left(8-\xi_{(m+1,n)}\right) r^{2(m+1+n)}+\frac{n}2\xi_{(m+1,n)}r^{2(m+n)}\right)\sin(n\theta)\notag\\
&\hspace{3em}\quad-\left(\frac{n}4\left(8-\xi_{(k+1,n)}\right)r^{2(k+1+n)}+\frac{n}2\xi_{(k+2,n)}r^{2(k+n)}\right)\sin(n\theta);\label{poiss_c0km}\\
  & \fkB(1-r^2) (\xi_{(m+1,n)}\fks_{(m,n)}-\xi_{(k+1,n)}\fks_{(k,n)})\notag\\
  &\hspace{3em}=-\left(\frac{n}4\left(8-\xi_{(m+1,n)}\right) r^{2(m+1+n)}+\frac{n}2\xi_{(m+1,n)}r^{2(m+n)}\right)\cos(n\theta)\notag\\
&\hspace{3em}\quad+\left(\frac{n}4\left(8-\xi_{(k+1,n)}\right)r^{2(k+1+n)}+\frac{n}2\xi_{(k+1,n)}r^{2(k+n)}\right)\cos(n\theta)\label{poiss_s0km}
  \end{align}   
  \end{subequations}
Furthermore, we have that~$8-\xi_{(m+1,n)}<0$, for all~$(n,m)\in\bbN_+\times\bbN$.
    \end{lemma}   
 The proof is given in the Appendix, Section~\ref{Apx:proofL:poiss_cs0km}.

  \subsection{Poisson bracket. Involving generators with nonzero periodic frequency.} 
  We consider~\eqref{generator} with the generators in~$\{\fkc_{(0,k)},\fkc_{(2,k)}\}\subset\clS^0$, with~$k\in\{1,2\}$. 
  
The  following result tells us how to generate  factors~$\sin((n+k)\theta)$, from available lower frequency factors~$\cos(k\theta)$ and~$\cos(n\theta)$.
 \begin{lemma}\label{L:poiss(n+k)_cc_gen}
 Let us denote, for~$(n,k)\in\bbN_+\times\bbN_+$, 
\begin{align}
\psi_k&\coloneqq\xi_{(2,k)}\fkc_{(1,k)}-\xi_{(1,k)}\fkc_{(0,k)}\quad \mbox{and}\quad
\phi_{(m,n)}\coloneqq\xi_{(m+1,n)}\fkc_{(m,n)}-\xi_{(1,n)}\fkc_{(0,n)}.\notag
\end{align}
Then, we have that
 \begin{align}\notag
 & r^{-2(n+1)}\clB(\psi_k) \phi_{(m,n)}=Q^{(n,k)}_{+,m}(r^2)\sin((n+k)+Q^{(n,k)}_{-,m}(r^2)\sin((n-k)\theta)
  \end{align}   
 with polynomial functions~$Q^{(n,k)}_{+,m}$ and~$Q^{(n,k)}_{-,m}$ of degree~$m+1$ as follows,
 \begin{align}\notag
Q^{(n,k)}_{+,m}(s)&={C}_{+,1,m}^{(n,k)}s^{m+1}+C_{+,0,m}^{(n,k)}s^{m}+C_{+,1}^{(n,k)}s+C_{+,0}^{(n,k)},\notag\\
Q^{(n,k)}_{-,m}(s)&={C}_{-,1,m}^{(n,k)}s^{m+1}+C_{-,0,m}^{(n,k)}s^{m}+C_{-,1}^{(n,k)}s+C_{-,0}^{(n,k)},\notag
\intertext{with coefficients}
{C}_{+,1,m}^{(n,k)}&\coloneqq  (2n-km)\xi_{(m+1,n)}+(k(m+1)-n)\xi_{(2,k)},\notag\\
C_{+,0,m}^{(n,k)}&\coloneqq  (km-n)\xi_{(m+1,n)}-k(m+1)\xi_{(1,k)},\notag\\
C_{+,1}^{(n,k)}&\coloneqq-2n\xi_{(1,n)}+(n-k)\xi_{(2,k)},\notag\\
C_{+,0}^{(n,k)}&\coloneqq n\xi_{(1,n)}+k\xi_{(1,k)},\notag
\intertext{and}
{C}_{-,1,m}^{(n,k)}&\coloneqq  (k(m+2n)+2n)\xi_{(m+1,n)}-(k(m+1+2n)+n)\xi_{(2,k)},\notag\\
C_{-,0,m}^{(n,k)}&\coloneqq  -(k(m+2n)+n)\xi_{(m+1,n)}+k(m+1+2n)\xi_{(1,k)},\notag\\
C_{-,1}^{(n,k)}&\coloneqq-2n(k+1)\xi_{(1,n)}+(k(2n+1)+n)\xi_{(2,k)},\notag\\
C_{-,0}^{(n,k)}&\coloneqq n(2k+1)\xi_{(1,n)}-k(2n+1)\xi_{(1,k)},\notag
   \end{align}     
where the function~$\xi$ is as in~\eqref{xi_mn}.
\end{lemma} 
The proof  is given in the Appendix, Section~\ref{Apx:proofL:poiss(n+k)_cc_gen}. 

Analogously, we  can generate  factors~$\cos((n+k)\theta)$, from available lower frequency factors~$\cos(k\theta)$ and~$\sin(n\theta)$.
 \begin{lemma}\label{L:poiss(n+k)_cs_gen}
Let us denote, for~$(n,k)\in\bbN_+\times\bbN_+$,
\begin{align}
\psi_k&\coloneqq\xi_{(2,k)}\fkc_{(1,k)}-\xi_{(1,k)}\fkc_{(0,k)}\quad \mbox{and}\quad
\varphi_{(m,n)}\coloneqq\xi_{(m+1,n)}\fks_{(m,n)}-\xi_{(1,n)}\fks_{(0,n)}.\notag
\end{align}
Then, we have that
 \begin{align}\notag
 & r^{-2(n+1)}\fkB(\psi_k) \varphi_{(m,n)}=-Q^{(n,k)}_{+,m}(r^2)\cos((n+k)-Q^{(n,k)}_{-,m}(r^2)\cos((n-k)\theta)
  \end{align}   
with the polynomial functions~$Q^{(n,k)}_{+,m}$ and~$Q^{(n,k)}_{-,m}$ as in Lemma~\ref{L:poiss(n+k)_cc_gen}.
\end{lemma} 
The proof  is given in the Appendix, Section~\ref{Apx:proofL:poiss(n+k)_cs_gen}.

 \begin{lemma}\label{L:3vec-(n+1)}
For~$n\ge3$, the polynomials~$Q^{(n,k)}_{+,m}$ 
as in Lemma~\ref{L:poiss(n+k)_cc_gen} satisfy
 \begin{align}
 &\linspan\{Q^{(n+1-k,k)}_{+,m}\mid (k,m)\in\{1,2\}\times\{1,2\}\}=\linspan\{r^{j}\mid j\in\{0,1,2,3\}\}.\notag
  \end{align}   
   \end{lemma} 
  The proof is given in the Appendix, Section~\ref{Apx:proofL:3vec-(n+1)}.

 \begin{lemma}\label{L:n=k}
 For every~$n\ge1$, the leading coefficient of the polynomials~$Q^{(n,k)}_{-,m}$ 
as in Lemma~\ref{L:poiss(n+k)_cc_gen}  satisfies, for~$k=n$,
 \begin{align}
  &{C}_{-,1,m}^{(n,n)}>0,\quad\mbox{for all}\quad m>1.
  \end{align}   
 \end{lemma} 
  The proof is give  in the Appendix, Section~\ref{Apx:proofL:n=k}.

\subsection{Proof of Theorem~\ref{T:saturating-vort}: conclusion}\label{sS:proofT:saturating-vort}
We start by showing that the functions in~\eqref{cs_mn} 
are elements of some subspace in the sequence in~\eqref{seqclG}, that is, we shall show that 
\begin{equation}\label{fkcs-in-seqclG}
\stepcounter{equation}\tag{\theequation:goal}
 \clS^\infty\coloneqq\{\fkc_{(m,n)}\mid (m,n)\in\bbN\times\bbN\}{\,\textstyle\bigcup\,} \{\fks_{(m,n)}\mid (m,n)\in\bbN\times\bbN_+\}\subset{\textstyle\bigcup_{j\in\bbN}}\clS^j.
\end{equation}

Following~\eqref{seqclS}, we set~$\clS^0=\linspan E$ and observe that
\begin{equation}\label{fkcs-in-seqclG-0}
 E= \{\fkc_{(k,0)}\mid 0\le k\le4\}{\,\textstyle\bigcup\,} \{\fkc_{(m,n)},\fks_{(m,n)}\mid (m,n)\in\{0,1\}\times\{1,2,3\}\}\subset\clS^0.
\end{equation}

Note that~$\{1-r^2,\xi_{(2,1)}\fkc_{(1,1)}-\xi_{(1,1)}\fkc_{(1,1)}\}\in\clS^0$

Using~\eqref{poiss_c0km} and~\eqref{poiss_s0km}, with $(k,m)=(0,1)$ and~$n\in\{1,2,3\}$  we can see that
\begin{align}
 4\fkB(1-r^2) (\xi_{(2,n)}\fkc_{(1,n)}-\xi_{(1,n)}\fkc_{(0,n)})&=\left(8-\xi_{(2,n)}\right) \fks_{(2,n)}+\Phi^0\notag\\
 -4\fkB(1-r^2) (\xi_{(2,n)}\fks_{(1,n)}-\xi_{(1,n)}\fks_{(0,n)})&=\left(8-\xi_{(2,n)}\right) \fkc_{(2,n)}+\Psi^0\notag
  \end{align}
with~$\left(8-\xi_{(2,n)}\right)\ne0$ and~$\{\Phi^0,\Psi^0\}\subset\clS^0$. 
Note also that, due to~\eqref{cs-in-W12}, the selected functions~$\xi_{(2,n)}\fkc_{(1,n)}-\xi_{(1,n)}\fkc_{(0,n)}$ and~$\xi_{(2,n)}\fks_{(1,n)}-\xi_{(1,n)}\fks_{(0,n)}$ are in~$\clC^{1,1}(\overline\clD)$.
Thus, we can conclude that
~$\{\fks_{(2,n)},\fkc_{(2,n)}\}\subset\clS^1$. 

Next, we proceed by induction. If~$\{\fks_{(\underline m,n)},\fkc_{(\underline m,n)}\}\subset \clS^{\underline m-1}$, for all~$2\le\underline m\le\overline m$ and~$n\in\{1,2,3\}$, then using again~\eqref{poiss_c0km} and~\eqref{poiss_s0km}, now with $(k,m)=(0,\overline m)$, we find that 
\begin{align}
  4\fkB(1-r^2) (\xi_{(\overline m+1,n)}\fkc_{(\overline m,n)}-\xi_{(1,n)}\fkc_{(0,n)})&=\left(8-\xi_{(\overline m+1,n)}\right) \fks_{(\overline m+1,n)}+\Phi^{\overline m}\notag\\
    -4\fkB(1-r^2) (\xi_{(\overline m+1,n)}\fks_{(\overline m,n)}-\xi_{(1,n)}\fks_{(0,n)})&=\left(8-\xi_{(\overline m+1,n)}\right) \fkc_{(\overline m+1,n)}+\Phi^{\overline m}\notag
  \end{align}
with~$\left(8-\xi_{(\overline m+1,n)}\right)\ne0$ and~$\{\Phi^{\overline m},\Psi^{\overline m}\}\subset\clS^{\overline m}$. Again, since~$n\ge1$ and due to~\eqref{cs-in-W12}, 
we have that~$\{\fks_{(\overline m+1,n)},\fkc_{(\overline m+1,n)}\}\subset\clS^{\overline m+1}$ and, by induction
we can conclude that
\begin{equation}\label{fkcs-in-seqclG-n123}
 \{\fks_{(m,n)},\fkc_{(m,n)}\mid m\in\bbN,\quad n\in\{1,2,3\}\}\subset{\textstyle\bigcup_{j\in\bbN}}\clS^j.
\end{equation}

Let us now assume that for some~$\overline n\ge 3$ we have that
\begin{equation}\label{fkcs-in-seqclG-nHyp}
\stepcounter{equation}\tag{\theequation:hyp}
 \{\fks_{(m,n)},\fkc_{(m,n)}\mid m\in\bbN,\quad 1\le n\le \overline n\}\subset{\textstyle\bigcup_{j\in\bbN}}\clS^j.
\end{equation}

By using Lemma~\ref{L:poiss(n+k)_cc_gen} with~$(k,m)\in\{1,2\}\times\{1,2\}$, together with Lemma~\ref{L:3vec-(n+1)} and~\eqref{fkcs-in-seqclG-nHyp}, we can conclude that 
\begin{subequations}\label{fkcs-in-seqclG-n+1}
\begin{equation}
 \{\fks_{(m,\overline n+1)}\mid m\in\{0,1,2,3\}\}\subset{\textstyle\bigcup_{j\in\bbN}}\clS^j,
 \end{equation}
where we have used again~\eqref{cs-in-W12} since~$m=2$.
 
 Similarly, combining Lemma~\ref{L:poiss(n+k)_cs_gen} with~$(k,m)\in\{1,2\}\times\{1,2\}$, together with Lemma~\ref{L:3vec-(n+1)} and~\eqref{fkcs-in-seqclG-nHyp}, we obtain
 \begin{equation}
 \{\fkc_{(m,\overline n+1)}\mid m\in\{0,1,2,3\}\}\subset{\textstyle\bigcup_{j\in\bbN}}\clS^j.
 \end{equation}
\end{subequations}
Now, we can use
~\eqref{poiss_c0km} and~\eqref{poiss_s0km} as above, to conclude that
\begin{equation}\label{fkcs-in-seqclG-nThes}
 \{\fks_{(m,\overline n+1)},\fkc_{(m,\overline n+1)}\mid m\in\bbN\}\subset{\textstyle\bigcup_{j\in\bbN}}\clS^j.
\end{equation}
That is, we have that the hypothesis~\eqref{fkcs-in-seqclG-nHyp} implies~\eqref{fkcs-in-seqclG-nThes}. Combining this implication with~\eqref{fkcs-in-seqclG-n123}, by induction, we obtain
\begin{equation}\label{fkcs-in-seqclG-n>1}
 \{\fks_{(m,n)},\fkc_{(m,n)}\mid m\in\bbN, n\in\bbN_+ \}\subset{\textstyle\bigcup_{j\in\bbN}}\clS^j.
\end{equation}

Let us now assume that we have that
 \begin{equation}\label{fkcs-in-seqclG-n=0hyp}
 \stepcounter{equation}\tag{\theequation:hyp}
 \fkc_{(m,0)}\in{\textstyle\bigcup_{j\in\bbN}}\clS^j,\quad\mbox{for all}\quad m\le\overline m, \quad\mbox{for some given}\quad\overline m\ge4. 
  \end{equation} 
Note that, by~\eqref{fkcs-in-seqclG-0} we have that~\eqref{fkcs-in-seqclG-n=0hyp} holds with~$\overline m=4$.
Next, we use Lemma~\ref{L:poiss(n+k)_cs_gen} with~$n=k=1$. 
Then, using~\eqref{fkcs-in-seqclG-n>1}, we find that
 \begin{align}\notag
 & \fkB(\psi_1) \varphi_{(\overline m-2,1)}=-Q^{(1,1)}_{-,\overline m-2}(r^2)r^{4}\in{\textstyle\bigcup_{j\in\bbN}}\clS^j.
  \end{align}  
Since~$Q^{(1,1)}_{-,\overline m-2}$ is a polynomial of degree~$\overline m-1$, with leading coefficient~${C}_{-,1,\overline m-2}^{(1,1)}$, by
\eqref{fkcs-in-seqclG-0} we obtain that
 $ -{C}_{-,1,\overline m-2}^{(1,1)}r^{2(\overline m-1)+4}\in{\textstyle\bigcup_{j\in\bbN}}\clS^j$ and, by Lemma~\ref{L:n=k}, it follows that
  \begin{align}\label{fkcs-in-seqclG-n=0thes}
 & \fkc_{(\overline m+1,0)}=r^{2(\overline m+1)}\in{\textstyle\bigcup_{j\in\bbN}}\clS^j,
  \end{align} 
  where we have used again~\eqref{cs-in-W12} since~$n=1$.
 That is, the hypothesis~\eqref{fkcs-in-seqclG-nHyp} implies~\eqref{fkcs-in-seqclG-nThes}. By induction, using~\eqref{fkcs-in-seqclG-0}, it follows that
  \begin{equation}
   \fkc_{(m,0)}\in{\textstyle\bigcup_{j\in\bbN}}\clS^j,\quad\mbox{for all}\quad m\in\bbN,
  \end{equation} 
which conbined with~\eqref{fkcs-in-seqclG-n>1}, give us that the inclusion in~\eqref{fkcs-in-seqclG} holds true.

To finish the proof, it remains to show that
${\textstyle\bigcup_{j\in\bbN}}\clS^j$ is dense in~$W^{-1,2}$. Note that, since~$H$ is dense in~$W^{-1,2}$, it is sufficient to show that
the linear space~$\clS^\infty$ in~\eqref{fkcs-in-seqclG} is dense in~$H$. For this purpose, let us fix an arbitrary function~$h\in (\clS^\infty)^{\perp,H}$.
Recall that an orthogonal basis in~$H$ of eigenfunctions of the Laplacian can be found in the form
\begin{subequations}\label{Eigf-Bessel}
\begin{align}
&\overline e_{j,n}^c(r,\theta)=J_{n}(\rho_{j,n}r)\cos(n\theta),\quad (j,n)\in\bbN_+\times\bbN\\
&\overline e_{j,n}^s(r,\theta)=J_{n}(\rho_{j,n}r)\sin(n\theta),\quad(j,n)\in\bbN_+\times\bbN_+
\end{align}
\end{subequations}
where the $J_{n}\colon[0,+\infty)\to\bbR$ are Bessel functions solving
  \begin{equation}\label{BesselEq}
   s^2(\tfrac{\rmd}{\rmd s})^2 z +s\tfrac{\rmd}{\rmd s} z +(s^2-n^2)z=0,\qquad s\ge0.
  \end{equation} 
and the~$\rho_{j,n}>0$ are the corresponding positive zeros,~$J_{n}(\rho_{j,n})=0$ (cf.\cite[Ch.~II, Sects.~2.1.3, Eq.~(1)]{Watson66}). For the completeness, in~$L^2((0,1))$, of each of the families~$\{J_{n}(\rho_{j,n}r)\mid j\in\bbN_+\}$ we refer to~\cite[Ch.~XVIII, Sect.~18.2.4-5]{Watson66}. Finally, note that from~\eqref{A-pol}, we find
\begin{align}
\Delta\overline e_{j,n}^c(r,\theta)&=\left(\rho_{j,n}^2(\tfrac{\rmd}{\rmd s})^2 J_{n}(\rho_{j,n}r)+r^{-1}\rho_{j,n}\tfrac{\rmd}{\rmd s} J_{n}(\rho_{j,n}r) -r^{-2}n^2 J_{n}(\rho_{j,n}r)\right)\cos(n\theta),
\end{align}
and by~\eqref{BesselEq}, with~$\xi=r\rho_{j,n}$,
\begin{align}
r^2\Delta\overline e_{j,n}^c(r,\theta)&=\left(\xi^2(\tfrac{\rmd}{\rmd s})^2 J_{n}(\xi)+\xi\tfrac{\rmd}{\rmd s} J_{n}(\xi) -n^2 J_{n}(\xi)\right)\cos(n\theta)\notag\\
&=(-(\xi^2-n^2) -n^2) J_{n}(\xi)\cos(n\theta)=-\xi^2\overline e_{j,n}^c(r,\theta)=-r^2\rho_{j,n}^2\overline e_{j,n}^c(r,\theta).\notag
\end{align}
Similarly, we can find that~$r^2\Delta\overline e_{j,n}^s(r,\theta)=-r^2\rho_{j,n}^2\overline e_{j,n}^c(r,\theta)$, giving us that the corresponding eigenvalue of~$A=-\Delta$ is~$\rho_{j,n}^2$,
\begin{align}
&A\overline e_{j,n}^c(r,\theta)=\rho_{j,n}^2\overline e_{j,n}^c(r,\theta),\quad (j,n)\in\bbN_+\times\bbN\notag\\&A\overline e_{j,n}^s(r,\theta)=\rho_{j,n}^2\overline e_{j,n}^s(r,\theta),\quad(j,n)\in\bbN_+\times\bbN_+.\notag
\end{align}

Since the eigenfunctions in~\eqref{Eigf-Bessel} form a complete system, we  can write
\begin{equation}\notag
h=h(r,\theta)={\textstyle\sum\limits_{n\in\bbN}}h_{\rmc,n}(r)\cos(n\theta)+{\textstyle\sum\limits_{n\in\bbN_+}}h_{\rms,n}(r)\sin(n\theta).
\end{equation}
Then, from~$h\in (\clG^\infty)^{\perp,H}$ we find that
\begin{align}
0&=(h,\fkc_{(m,n)})_H=\int_0^1r\,\rmd r\int_0^{2\pi} h(r,\theta)\fkc_{(m,n)}(r,\theta)\,\rmd\theta\notag\\
&=\int_0^1r\,\rmd r\int_0^{2\pi}h_{\rmc,n}(r)r^{2(n+m)}\cos^2(n\theta)\,\rmd\theta\notag
\intertext{and, analogously,}
0&=(h,\fks_{(m,n)})_H=\int_0^1r\,\rmd r\int_0^{2\pi}h_{\rms,n}(r)r^{2(n+m)}\sin^2(n\theta)\,\rmd\theta\notag
\end{align}
which implies that
\begin{subequations}\label{orth-evenmonom}
\begin{align}
0&=\int_0^1 r^{2m} h_{\rmc,n}(r)r^{2(n+1)}\,\rmd r,\quad\mbox{for all}\quad n\in\bbN\mbox{ and all } m\in\bbN;\\
0&=\int_0^1 r^{2m} h_{\rms,n}(r)r^{2(n+1)}\,\rmd r,\quad\mbox{for all}\quad n\in\bbN_+\mbox{ and all }m\in\bbN.
\end{align}
\end{subequations}
Recalling that the monomials in~$\{r^m\mid m\in\bbN\}$ form a basis in~$L^2(-1,1)$ we find that
the monomials in~$\{r^{2m}\mid m\in\bbN\}$ form a basis in~$L^2(0,1)$ (e.g., note that a function~$h\in L^2(0,1)$ can be extended to an even function in~$(-1,1)$ by setting~$h(-r)\coloneqq h(r)$, and this extension is clearly orthogonal to the odd degree monomials).

Since~$\{r^{2m}\mid m\in\bbN\}$ is a basis, by~\eqref{orth-evenmonom}, the functions $h_{\rmc,n}(r)r^{2(n+1)}$ and~$h_{\rms,n}(r)r^{2(n+1)}$ must vanish, which leads to~$h=0$ and~$(\clS^\infty)^{\perp,H}=\{0\}$, that is,~$\clS^\infty$ is dense in~$H$.
\qed

\section{Final remarks}\label{S:final-rmks}

\subsection{On the smoothness of the spatial domain}\label{sS:rmk-smooth}
This manuscript shows the approximate controllability of~2D Euler equations in the case where the unit disk is the spatial domain.
Previous analogue results in the literature are restricted to boundaryless spatial domains, as a Torus~$\bbT^2$~\cite{AgraSary05,AgraSary06} or  a Sphere~$\bbS^2$~\cite{AgraSary08}.

We follow the Agrachev--Sarychev approach, and show that approximate controllability follows from the existence of saturating set~$G_0$ included in the set of actuators~$U$; see~\eqref{U}.
We relax the notions of saturating set in the literature, which allows us to find an example of a saturating set in the case the spatial domain is the unit disk.

By considering $\clC^\infty$-smooth domains~$\Omega$, we can directly apply available results in the literature, stated for such domains. The lowest possible regularity of the boundary~$\p\Omega$ of~$\Omega$, under which Theorem~\ref{T:KatoShiri} holds true is not the focus of this work. In any case, we would like to mention that Theorem~\ref{T:KatoShiri}  will likely hold for  ``regular enough'' domains, namely, with $\clC^{m}$-smooth boundary with~$m\in\bbN$ ``large enough'', for example, as stated in~\cite[Intro., discussion preceding Eq.~(E)]{Kato67}. In that case, our approximate controllability result in Theorem~\ref{T:approxcontrol} will also hold for such~$\clC^{m}$-domains.

It would be interesting to know whether Theorems~\ref{T:KatoShiri} and~\ref{T:approxcontrol} also hold true for the particular case of a rectangular domain~$\Omega\subset\bbR^2$. Because in this case we have a saturating set as given in~\cite[Sect.~6.3]{Rod-Thesis08}.

We focus on ${\rm 2D}$ bounded domains~$\Omega\subset\bbR^2$.  It is likely that the result can be extended to ${\rm 2D}$ (compact) Riemannian manifolds with smooth boundary. In this case we will have a saturating set as in~\cite[Thm.~6.4.10]{Rod-Thesis08} for the case of the hemisphere~$\bbS_+^2\subset\bbR^3$.

\subsection{On approximate controllability in other norms.}\label{sS:rmk-appcontr-norm}
We have shown approximate controllability with respect to the norm of the pivot space~$H$ in~\eqref{H}; see Definition~\ref{D:acon-U}. A natural question is whether an analogous result holds for stronger norms, namely, when we replace the condition~$\norm{y(T)-y_\tta}{H}\le \varepsilon$ by~$\norm{y(T)-y_\tta}{\clX}\le \varepsilon$ in Definition~\ref{D:acon-U}. For example, from Theorem~\ref{T:KatoShiri}, candidates for~$\clX$ could be spaces as~$\clX=H\bigcap\clC^{m}(\overline\Omega)$, $0\le m\le 2$. The proof of such a result will likely require considerably different arguments. Other candidates for~$\clX$ could be Sobolev subspaces (cf.~\cite{Nersisyan10} for a 3D boundaryless torus~$\Omega=\bbT^3$ as spatial domain; see also~\cite[Sect.~15.2]{EbinMars70} for well-posedness results concerning initial data in Sobolev subspaces for general smooth bounded domains/manifolds and~\cite{LinSunLiuZhang24} for an unbounded cylindrical domain).

\subsection{On approximate controllability by spatially localized actuators}
The set~$G_0$ in~\eqref{satset-EG0} consists of actuators~$\Phi_i$ whose supports~$\supp(\Phi_i)=\overline\Omega$ cover the entire domain. We would like to recall a question introduced by Agrachev in~\cite[Sect.~VII]{Agrachev13-arx}: are the Euler equations approximately controllable by a finite number of actuators all supported in an apriori given  small subdomain,~$\supp(\Phi_i)\subset\overline\omega\subset\Omega$? A different question without fixing~$\omega$ a priori could be:  if we fix a priori the total volume covered by the supports of the actuators as~$\rho\vol(\Omega)$ with~$0<\rho<1$, can we still find actuators so that  the Euler equations are approximately controllable?

\subsection{On Euler  equations versus Navier--Stokes equations}\label{sS:rmk-NS}
We derived Theorem~\ref{T:approxcontrol} stating approximate controllability for Euler  equations~\eqref{eul-sys}, by means of smooth control inputs~$u\in\clC^\infty([0,T];\bbR^M)$ based on a set~$U$ of~$M$ actuators as in~\eqref{U}, constructed from a saturating set as in Definition~\ref{D:saturVlin}. A question arises on whether an analogous approximate controllability result holds for Navier--Stokes equations~\eqref{ns-sys} with analogue control inputs and saturating sets. Such a result seems plausible, but there are details to be checked and new difficulties to be overcome. The investigation of these details is postponed for a future work. The reason is that, though the main steps can likely  be followed, the mathematical setting and analysis will be different. Solutions evolving in H\"older spaces, as in Theorem~\ref{T:KatoShiri}, are convenient to handle the Euler equations. Instead, for Navier--Stokes equations it is more natural to work with variational weak/strong solutions evolving in Sobolev spaces, as in~\cite{Temam01}; see also the remark in~\cite[Sect.~15.6(i)]{EbinMars70} suggesting that looking for Navier--Stokes solutions evolving  in H\"older spaces, will require considerable work.

Due to the parabolic-like nature of Navier--Stokes equations we will likely be able to take less regular initial conditions and external forces. On the other hand, we will need to handle the additional boundary conditions and the additional term given by the Stokes operator~$ A\coloneqq -\Pi\Delta$, depending on the given boundary conditions. Just to give an illustrative example, following the arguments in Section~\ref{sS:imit-comp2} we will arrive to the analogue of~\eqref{dyn.Zimity2}, 
\begin{align}
\dot Z&=-\nu A(Z-\varkappa_2)-B(Z,Z)+\clB(\varkappa_2-\underline y)Z+\clB(\underline y)\varkappa_2+h_{\beta,K}\notag\\
&=-\nu AZ-B(Z,Z)+\clB(\varkappa_2-\underline y)Z+\clB(\underline y)\varkappa_2+h_{\beta,K}+\nu A\varkappa_2,\label{dyn.Zimity2-NS}
\end{align}
with~$\overline z(t)=Z(t)-\varkappa_2(t)\in\rmD(A)$, and since~$\varkappa_2(t)\in H\subset L^2(\Omega)^2$ we meet a technical issue, namely, that we cannot guarantee more regularity for~$Z$ than~$Z(t)\in H$. In particular, the classical weak/strong energy estimates are not valid for~$Z$ in~\eqref{dyn.Zimity2-NS}. Note also that the extra forcing~$\nu A\varkappa_2\in\rmD(A^{-1})$ does not have the regularity~$\nu A\varkappa_2\in\rmD(A^{-\frac12})$ required by weak solutions, when~$\varkappa_2(t)\in H\setminus\rmD(A^{\frac12})$. This means that we will need some extra nontrivial argument at this point. We can avoid this regularity issue if we can find a saturating set with more regular actuators and satisfying the boundary conditions; see~\cite[Thm.~2.2 and~2.4]{Shirikyan06} \cite[Thm.~2.2 and Props.~3.1--3.3]{Shirikyan07} where the actuators are assumed to be in the domain of~$A$, $\rmD(A) \subset\{h\in W^{2,2}\mid\; h\rest{\partial\Omega}=0\}\subset H$ (under no-slip boundary conditions); see~\cite[Sect.~1.1]{Shirikyan06}~\cite[Sect.~1.2]{Shirikyan07}. Here we must say that the results in~\cite{Shirikyan06,Shirikyan07} are focused on the more demanding~3D case, where the long-time existence and uniqueness of the controlled solutions has to be guaranteed if the time-horizon~$T$ is large. Anyway, the results in~\cite{Shirikyan06,Shirikyan07} can be applied to the 2D Navier--Stokes equations as well, where well-posedness of weak variational solutions is known.

Can we find a saturating set as required in~\cite{Shirikyan06} with all its elements in~$\rmD(A)$?
For no-slip boundary conditions, this could be an interesting work for a future work, regardless of the bounded spatial domain considered, either 2D or 3D.

Note that the subset~$G^0$ in~\eqref{satset-vecfields} is not contained in~$\rmD(A)$, in the unit disk under no-slip boundary conditions, because~$h_1=\curl^*(-\Delta_\rmd)^{-1}1\in G^0\setminus\rmD(A)$. Indeed,~$h_1=\frac12(-x_2,x_1)$ does not satisfy the no-slip boundary conditions. 
The subset~$G^0$ in~\eqref{satset-vecfields} is also not contained in~$\rmD(A)$, in the unit disk under Lions boundary conditions, because~$\curl h_1=\curl \curl^*(-\Delta_\rmd)^{-1}1=1$ does not vanish on~$\partial\clD$.

Are the Navier--Stokes equations in the unit disk~$\clD$ approximately controllable by means of the actuators in~$U\supset G^0$ as in~\eqref{satset-vecfields}? The discussion above tells us that the answer is not trivial. This could also be an interesting subject for future work.

In the case of a bounded smooth domain~$\Omega\subset\bbR^d$, finding examples of saturating sets leading to approximate controllability of Navier--Stokes equations is an interesting problem, regardless of the type of boundary conditions (either of no-slip or of slip type) and regardless of the dimension~$d\in\{2,3\}$ of the spatial domain. In this manuscript, for Euler equations (where the boundary conditions are less restrictive)  we provided an example in the case~$\Omega=\clD$ is the unit disk. It would also be interesting to know examples for other spatial domains as well, possibly leading to further insights on the approximate controllability of the equations of fluid mechanics. With this respect, we may ask the following question inspired by the discussion in~\cite[Sect.~9]{AgraSary08}: using the result for the unit disk, can an appropriate transformation/argument be found to show approximate controllability for ``generic'' simply connected smooth domains?

\appendix\normalsize
\section*{Appendix}
\setcounter{section}{1}

\subsection{Proof of Lemma~\ref{L:estTrilin-poisson}}\label{Apx:proofL:estTrilin-poisson}
Recalling~\eqref{Trilin-poisson} and~\eqref{B-formal}, we find
   \begin{align}
   \bfb(y,z,w)&=(B(y,z),w)_H=((y\cdot\nabla z_1, y\cdot\nabla z_2),w)_{L^2}=\sum_{i=1}^2(y,w_i\nabla z_i)_{L^2} \notag\\
   &=\sum_{i=1}^2\int_{\Omega}\left(\diver(w_i z_iy)-z_iy\cdot\nabla w_i-z_i w_i(\diver y)\right)\rmd\Omega=-\sum_{i=1}^2(y,z_i\nabla w_i)_{L^2}\notag\\
   &=-\bfb(y,w,z),\notag   
     \end{align} 
where we used~$(\diver v,1)_{L^2}=(\diver v,1)_{L^2(\Omega)}=(\bfn\cdot h,1)_{L^2(\p\Omega)}$ and~$y\in H$; see~\eqref{H}. In particular, $\bfb(y,z,z)=0$. Thus, \eqref{asymTrilin} holds true.

Next, we can estimate~$\norm{\bfb(y,z,w)}{\bbR}=\norm{\sum_{i=1}^2(y,w_i\nabla z_i)_{L^2}}{\bbR}$ as
\begin{align}
\norm{\bfb(y,z,w)}{\bbR}&\le C_1\norm{y}{L^2}\norm{z}{\clC^{1}(\overline\Omega)}\norm{w}{L^2},\notag\\
\norm{\bfb(y,z,w)}{\bbR}&\le C_2\norm{y}{L^4}\norm{\nabla z}{L^2}\norm{w}{L^4}\le C_3\norm{y}{L^4}\norm{z}{V}\norm{w}{L^4},\notag
\end{align}
where we used Lemma~\ref{L:vect-vort-est}, with the norm in~$V= W^{1,2} \bigcap H$ associated to the scalar product in~\eqref{scalarprodV}. Finally, we  find
\begin{align}
&\norm{\bfb(y,z,w)+\bfb(z,y,w)}{\bbR}=\norm{\langle B(y,z)+B(z,y),w\rangle_{V,V'}}{\bbR}\le C_4\norm{B(y,z)+B(z,y)}{V}\norm{w}{V'}\notag\\
&\hspace{1em}\le C_5\left(\norm{B(y,z)+B(z,y)}{H}+\norm{(\curl (B(y,z)+B(z,y))}{L^2}\right)\norm{w}{V'}\notag\\
&\hspace{1em}= C_5\left(\norm{B(y,z)+B(z,y)}{H}+\norm{y\cdot\nabla(\curl z)+z\cdot\nabla(\curl y)}{L^2}\right)\norm{w}{V'}\notag\\
&\hspace{1em}\le  C_6\left(\norm{y}{L^2}\norm{z}{\clC^1}+\norm{z}{L^2}\norm{y}{\clC^1}+\norm{y}{\clC^0}\norm{(\curl z}{\clC^1}+\norm{z}{\clC^0}\norm{(\curl y}{\clC^1}\right)\norm{w}{V'}\notag\\
&\hspace{1em}\le  C_7\norm{(y,\nabla y)}{\clC^1}\norm{(z,\nabla z)}{\clC^1}\norm{w}{V'}.\notag
\end{align}
We can finish the proof by taking~$C_\bfb=\max\{C_1,C_3,C_7\}$.
\qed

\subsection{Proof of Lemma~\ref{L:findimGj}}\label{Apx:proofL:findimGj}
First of all~$\clG^0$ is finite dimensional, since~$G_0$ if finite; then we can fix a basis~$\{\Psi_k^0\mid 1\le k\le \overline m_{0}\}$ for~$\clG^0$. We proceed by Induction. Let~$j\ge0$ and assume that~$\{\Psi_k^j\mid 1\le k\le \overline m_{j}\}$ is a basis for~$\clG^j$. Given an arbitrary~$h\in\clG^{j+1}$ we can write
\begin{equation}\notag
h=p+q\quad\mbox{with}\quad p\in\clG^j\mbox{ and }q\in\clB_\clH(\clG^0)\clG^j.
\end{equation}
Observe that, recalling~\eqref{BZH-clH} and~\eqref{spanS},  we can write for an arbitrary~$\overline q\in\clB_\clH(\clG^0)\clG^j$,
\begin{equation}\notag
\overline q={\textstyle\sum\limits_{i=1}^N} c_i\clB(\zeta_i)(\eta_i)\quad\mbox{with}\quad (c_i,\zeta_i,\eta_i)\in\bbR\times\clG^0\times\clG^j\mbox{ and }\clB(\zeta_i)(\eta_i)\in\clH.
\end{equation}
Next, we see that for some scalars~$a_r$ and~$b_s$, with~$1\le r\le \overline m_{0}$ and~$1\le s\le \overline m_{j}$, we have
\begin{equation}\notag
\clB(\zeta_i)(\eta_i)={\textstyle\sum\limits_{k=1}^{\overline m_{0}}\sum\limits_{s=1}^{\overline m_{j}}}a_rb_s\clB(\Psi_r^0)\Psi_s^j\in \clZ_j\coloneqq\linspan\{\clB(\Psi_r^0)\Psi_s^j\mid 1\le r\le \overline m_{0},\;1\le s\le \overline m_{j}\}.
\end{equation}
Therefore, $\overline q\in\clZ_j$ and we can conclude that~$\clB_\clH(\clG^0)\clG^j\subseteq\clZ_j$, which implies that~$\clB_\clH(\clG^0)\clG^j$ is finite dimensional. Let~$\widehat n_j=\dim\clB_\clH(\clG^0)\clG^j$. Necessarily, we can find a basis for~$\clB_\clH(\clG^0)\clG^j$ in the form~$\{\clB(\psi_i)(\phi_i)\mid 1\le i\le \widehat n_j\}$. Then, after eliminating the elements of this basis which are in~$\clG^j$, and assuming that those are the last ones, we conclude that the remaining elements~$\{\clB(\psi_i)(\phi_i)\mid 1\le i\le  n_j\}$ (an empty set if~$n_j=0$) satisfy
\begin{equation}\notag
\clG^{j}+\clB_\clH(\clG^0)\clG^j=\linspan\left(
\{\Psi_k^j\mid 1\le k\le \overline m_{j}\}{\,\textstyle\bigcup\,}\{\clB(\psi_i)(\phi_i)\mid 1\le i\le n_j\}
\right).
\end{equation}
We finish the proof by setting~$\Phi_k\coloneqq\Psi_k^j$ for $1\le k\le \overline m_{j}$, and~$\Phi_{\overline m_{j}+i}=\clB(\psi_i)(\phi_i)$ for $1\le i\le n_{j}$. 
\qed

\subsection{Proof of Lemma~\ref{L:intsincosKTheta}}\label{Apx:proofL:intsincosKTheta}  
We  write
\begin{align}
&{\textstyle\int\limits_0^T}\sin(\tfrac{\pi K s}{T})\Theta(s)\,\rmd s={\textstyle\int\limits_0^{t_1}}\sin(\tfrac{\pi K s}{T})\Theta(s)\,\rmd s
+{\textstyle\int\limits_{t_1}^{t_2}}\sin(\tfrac{\pi K s}{T})\Theta(s)\,\rmd s+{\textstyle\int\limits_{t_2}^T}\sin(\tfrac{\pi K s}{T})\Theta(s)\,\rmd s\notag
\end{align}
with~$0<t_1\coloneqq\tfrac{T}{2K}\le\tfrac{(2K-1)T}{2K}\eqqcolon t_2<T$.
Recalling that~$W^{1,1}(0,T)\xhookrightarrow{}L^\infty(0,T)$ (cf.~\cite[Sect.8.2, Thm.~8.8]{Brezis11}),  with~$C\coloneqq\norm{\Id}{\clL(W^{1,1}(0,T),L^\infty(0,T))}$, we obtain
\begin{align}
&\pm{\textstyle\int\limits_0^T}\sin(\tfrac{\pi K s}{T})\Theta(s)\,\rmd s\le\tfrac{T}{2K}\norm{\Theta}{L^\infty(0,T)}
\pm{\textstyle\int\limits_{t_1}^{t_2}}\sin(\tfrac{\pi K s}{T})\Theta(s)\,\rmd s+\tfrac{T}{2K}\norm{\Theta}{L^\infty(0,T)}\notag\\
&\hspace{3em}\le\tfrac{T}{K}C\norm{\Theta}{W^{1,1}(0,T)}
\pm{\textstyle\int\limits_{t_1}^{t_2}}\tfrac{T}{\pi K}\cos(\tfrac{\pi K s}{T})\dot\Theta(s)\,\rmd s
\le\tfrac{T}{K}C\norm{\Theta}{W^{1,1}(0,T)}
+\tfrac{T}{\pi K}\norm{\dot\Theta}{L^1(0,T)}\notag\\
&\hspace{3em}\le \tfrac{(\pi C+1)T}{\pi}\norm{\Theta}{W^{1,1}(0,T)}K^{-1}.\notag
\end{align}
which gives us~\eqref{intsinKTheta}. Finally, we find  
\begin{align}
\pm{\textstyle\int\limits_0^T}\cos(\tfrac{\pi K s}{T})\Theta(s)\,\rmd s=\mp{\textstyle\int\limits_0^T}\tfrac{T}{\pi K}\sin(\tfrac{\pi K s}{T})\dot\Theta(s)\,\rmd s\le\tfrac{T}{\pi K}\norm{\dot\Theta}{L^1(0,T),}\le\tfrac{T}{\pi}\norm{\Theta}{W^{1,1}(0,T)}K^{-1},\notag
\end{align}
which gives us~\eqref{intcosKTheta}.
\qed

\subsection{Proof of Lemma~\ref{L:obliproj}}\label{Apx:proofL:obliproj}
Given~$M\in\bbN_+$, we find that~$P_{\clE_M^\top}^{\clU_M}P_{\clE_M}^{\clE_M^\top}=(\Id-P_{\clU_M}^{\clE_M^\top}) P_{\clE_M}^{\clE_M^\top}=(P_{\clE_M}^{\clE_M^\top}-P_{\clU_M}^{\clE_M^\top}) P_{\clE_M}^{\clE_M^\top}$, hence denoting by~$\bbB_1^{\clE_M}$ the unit ball in~$\clE_M$,
\begin{align}
\norm{P_{\clE_M^\top}^{\clU_M}P_{\clE_M}^{\clE_M^\top}}{\clL(H)}=
\sup_{\substack{q\in\bbB_1^{\clE_M}\setminus\{0\}}}
\norm{q-P_{\clU_M}^{\clE_M^\top}q}{H}. \notag
\end{align}
We know that~$\bbE_M=\{ e_i\mid 1\le i\le M\}$ is an orthonormal basis for~$\clE_M$.  Let~$U_M=\{ \Phi_i\mid 1\le i\le M\}$ be a basis for~$\clU_M$ and denote by~$\bbB_1^M$ the unit ball in~$\bbR^M$, and finally denote
\begin{align}\notag
\overline q\mapsto \bbE_M^\diamond \overline q=\sum_{i=1}^M\overline q_ie_i\quad\mbox{and}\quad \overline q\mapsto U_M^\diamond \overline q=\sum_{i=1}^M\overline q_i\Phi_i
\end{align}
mapping~$\overline q\in\bbR^M$, respectively, onto~$\clE_M$ and~$\clU_M$. Then, we find
\begin{align}
\norm{P_{\clE_M^\top}^{\clU_M}P_{\clE_M}^{\clE_M^\top}}{\clL(H)}&=\sup_{\overline q\in\bbB_1^M\setminus\{0\}}\norm{\bbE_M^\diamond\overline q-U_M^\diamond[(\bbE_M, U_M)_H]^{-1}\overline q}{H} \notag\\
&\le\sup_{\overline q\in\bbB_1^M\setminus\{0\}}\norm{U_M^\diamond(\Id-[(\bbE_M, U_M)_H]^{-1})\overline q}{H}+\sup_{\overline q\in\bbB_1^M\setminus\{0\}}\norm{(\bbE_M^\diamond-U_M^\diamond)\overline q}{H}, \notag
\end{align}
where we used $P_{\clU_M}^{\clE_M^\top}q= U_M^\diamond\Xi_M^{-1}(\bbE_M^\diamond)^{-1} q$, for~$q\in\clE_M$, where~$\Xi_M\coloneqq[(\bbE_M, U_M)_H]$ stands for the matrix with entry~$( e_i,\Phi_j)_H$ in the~$i$th row and~$j$th  column; see~\cite[Lem.~2.8]{KunRod19-cocv}. By a continuity argument~$\Xi_M$ converges to the identity matrix in~$\bbR^{M\times M}$ as~$\Phi_i\to e_i$ for all~$1\le i\le M$. Therefore, given~$\varrho>0$, by choosing~$ \Phi_i$ close enough to~$e_i$ for each~$1\le i\le M$ we will have that~$\norm{P_{\clE_M^\top}^{\clU_M}P_{\clE_M}^{\clE_M^\top}}{\clL(H)}<\varrho$. Finally, note that we can choose such~$\Phi_i$ so that~$\{\Phi_i\mid 1\le i\le M\}\subset\clG^{\overline\jmath}$, for some large enough~$\overline\jmath\in\bbN$, due to the density of the inclusion~$\bigcup_{j\in\bbN}\clG^{j}\subset H$.
\qed

\subsection{Proof of Lemma~\ref{L:saturVlin-vort}}\label{Apx:proofL:saturVlin-vort}
Given an arbitrary~$h\in V$ we have that~$\curl h\in L^2$ and,
using~\eqref{curl*form}, that
\begin{equation}\notag
\langle\curl h,g\rangle_{W^{-1,2},W^{1,2}_0}=(\curl h, g)_{L^2}\coloneqq(h, \curl^* g)_{H}
\end{equation}
holds true  for every~$(h,g)\in V\times W^{1,2}_0$. Thus, using~$V\xhookrightarrow{\rm d} H$, we can write
\begin{equation}
\langle\curl h,g\rangle_{W^{-1,2},W^{1,2}_0}\coloneqq(h, \curl^* g)_{H},\quad\mbox{for all }(h,g)\in H\times W^{1,2}_0.\label{curlL2}
\end{equation}
Hence, we can write that~$\curl h\in W^{-1,2}$ if~$h\in H$. Furthermore, we see that
\begin{subequations}\label{nH-ncV'}
\begin{equation}
\norm{\curl h}{W^{-1,2}}\le C_1\norm{h}{H}
\end{equation}
for some constant~$C_1>0$, and that for arbitrary~$h\in H$ and~$z\in H$, taking~$g=(-\Delta_\rmd)^{-1}\curl z$ we find~$\curl^*g=z$ and
\begin{align}
(h, z)_{H}&=(h, z)_{L^2\times L^2}=\langle \curl h, (-\Delta_\rmd)^{-1}\curl z\rangle_{W^{-1,2},W^{1,2}_0}\le C_2\norm{\curl h}{W^{-1,2}}\norm{\curl z}{W^{-1,2}}\notag\\
&\le C_2C_1\norm{\curl h}{W^{-1,2}}\norm{z}{H},\notag
\end{align}
hence we also have that, with~$C_3\coloneqq C_1C_2$,
\begin{equation}
\norm{h}{H}\le C_3\norm{\curl h}{W^{-1,2}}.
\end{equation}
\end{subequations}

From the inequalities in~\eqref{nH-ncV'} we can see that to finish the proof it is enough to show that the sequences in~\eqref{seqclG} and~\eqref{seqclS} are related by the identities
\begin{equation}\label{seq-fkGclG}
\clS^j=\curl\clG^j,\qquad j\in\bbN,
\end{equation}
which we can show by induction. Indeed, for~$j=0$, with~$G^0=\curl^* (-\Delta_\rmd)^{-1}S_0$ as in Definitions~\ref{D:saturVlin} and~\ref{D:saturVlin-vort}, we find that~$\clG^0=\linspan S_0=\linspan (\curl G^0)=\curl\linspan G^0=\curl\clG^0$. Hence~\eqref{seq-fkGclG} holds true for~$j=0$. 

Next, let us be given an arbitrary nonnegative integer~$n\in\bbN$ and assume that~\eqref{seq-fkGclG} holds true for all integers~$0\le j\le n$. We find
\begin{align} \label{curlfkG1}
\curl\clG^{n+1}&=\curl\clG^{n}+\curl\clB_\clH(\clG^0)(\clG^{n})
\end{align}
and, for the last term we can see that, with~$\clX\coloneqq\curl\clH$,
\begin{align} \label{curlfkG2}
\curl\clB_\clH(\clG^0)(\clG^{n})=\fkB_\clX(\clS^0)(\clS^{n}).
\end{align}
Indeed, with~$(h,g)=(\curl y,\curl z)$ we find
\begin{align}
\curl\clB(y)z&=\curl(\clB(y+z,y+z)-\clB(y,y)-\clB(z,z))\notag\\
&=\{A^{-1}(h+g),h+g\}_\fkp-\{ (-\Delta_\rmd)^{-1}h,h\}_\fkp-\{ (-\Delta_\rmd)^{-1}g,g\}_\fkp\notag\\
&=\{ (-\Delta_\rmd)^{-1}h,g\}_\fkp+\{ (-\Delta_\rmd)^{-1}g,h\}_\fkp=\fkB(h)g.\notag
\end{align}

By the induction hypothesis we obtain~\eqref{curlfkG2}.

By combining~\eqref{curlfkG1} with~\eqref{curlfkG2} and with the induction assumption, we obtain
\begin{align}\notag
\clS^{n+1}=\clS^{n}+\fkB_\clH(\clS^0)(\clS^{n})=\curl\clG^{n+1},
\end{align}
By induction we conclude that~\eqref{seq-fkGclG} holds for all~$j\in\bbN$, finishing the proof.
\qed

\subsection{Proof of Lemma~\ref{L:poiss_cs0km}}\label{Apx:proofL:poiss_cs0km}  
Let~$m\in\bbN$. Since~$\fkc_{(m,0)}=r^{2m}$ is independent of~$\theta$, by Corollary~\ref{C:invlap_cs} we have that~$A^{-1}\fkc_{(m,0)}$ is independent of~$\theta$ as well. Then, by~\eqref{poisson_b-pol}
it follows that~$\clB(1-r^2)\fkc_{(m,0)}=0$. Next, for all~$(n,m,k)\in\bbN_+\times\bbN\times\bbN$, we find that
\begin{equation}\label{iLap1-r2}
\mbox{for}\quad\varphi\coloneqq 1-r^2\quad\mbox{we have}\quad A^{-1}\varphi=\zeta\coloneqq\tfrac1{16}(r^4-1)-\tfrac14(r^2-1)
\end{equation}
and, by denoting
\begin{align}
\psi_\fkc\coloneqq\xi_{(m+1,n)}\fkc_{(m,n)}-\xi_{(k+1,n)}\fkc_{(k,n)},\notag\\
\psi_\fks\coloneqq\xi_{(m+1,n)}\fks_{(m,n)}-\xi_{(k+1,n)}\fks_{(k,n)},\notag
\end{align}
we also find, again by Corollary~\ref{C:invlap_cs},
\begin{align}
A^{-1}\psi_\fkc&= \eta\cos(n\theta)\quad\mbox{and}\quad A^{-1}\psi_\fks= \eta\sin(n\theta),\qquad\mbox{with}\quad\eta\coloneqq r^{2n}(r^{2(k+1)}-r^{2(m+1)}).\notag
\end{align}
The above identities, together with~\eqref{poisson_b-pol}, lead us to
 \begin{align}
 \clB(\varphi) \psi_\fkc&=\{A^{-1}\varphi,\psi_\fkc\}_\fkp-\{\varphi,A^{-1}\psi_\fkc\}_\fkp=\{\zeta,\psi_\fkc\}_\fkp-
 \{\varphi,\eta\cos(n\theta)\}_\fkp\notag
 \\
&= \tfrac{1}{r}(-n\tfrac{\p\zeta}{\p  r}\psi_\fks+n\eta\tfrac{\p \varphi}{\p  r}\sin(n\theta))=n((\tfrac12-\tfrac14r^2)\psi_\fks-2\eta\sin(n\theta))\notag\\
&=\tfrac{n}4\left((2-r^2)(\xi_{(m+1,n)}r^{2(m+n)}-\xi_{(k+1,n)}r^{2(k+n)})-8\eta\right)\sin(n\theta)\notag,
  \end{align}   
from which we obtain~\eqref{poiss_c0km}.
  Analogously,
 \begin{align}
 \clB(\varphi) \psi_\fks&=\{A^{-1}\varphi,\psi_\fks\}_\fkp-\{\varphi,A^{-1}\psi_\fks\}_\fkp=
\{\zeta,\psi_\fks\}_\fkp-\{\varphi,\eta\sin(n\theta)\}_\fkp\notag
 \\
&= \tfrac{1}{r}(n\tfrac{\p\zeta}{\p  r}\psi_\fkc-n\eta\tfrac{\p \varphi}{\p  r}\cos(n\theta))=n(-(\tfrac12-\tfrac14r^2)\psi_\fkc+2\eta\cos(n\theta))\notag\\
&=-\tfrac{n}4\left((2-r^2)(\xi_{(m+1,n)}r^{2(m+n)}-\xi_{(k+1,n)}r^{2(k+n)})-8\eta\right)\cos(n\theta),\notag
  \end{align}  
from which we obtain~\eqref{poiss_s0km}.  
Finally,  for~$(n,m)\in\bbN_+\times\bbN$, we find that 
     \begin{align}
 8-\xi_{(m+1,n)}=8-(4(m+1+n)^2-n^2)&=8-3n^2-4(m+1)^2 -8n(m+1)\notag\\
&< 8-4(m+1)^2-8(m+1)<0,\notag
 \end{align}
 which finishes the proof.
\qed

\subsection{Proof of Lemma~\ref{L:poiss(n+k)_cc_gen}}\label{Apx:proofL:poiss(n+k)_cc_gen}
Let us denote, for positive integers~$n\ge1$ and~$k\ge1$,
\begin{align}
\psi&\coloneqq\xi_{(2,k)}\fkc_{(1,k)}-\xi_{(1,k)}\fkc_{(0,k)},\qquad &\eta&\coloneqq r^{2k}(r^{2}-r^{4})\notag\\
\phi&\coloneqq\xi_{(m+1,n)}\fkc_{(m,n)}-\xi_{(1,n)}\fkc_{(0,n)},\qquad &\zeta&\coloneqq r^{2n}(r^{2}-r^{2(m+1)}).\notag
\end{align}

By Corollary~\ref{C:invlap_cs} we have that
\begin{align}
A^{-1}\psi&=\eta\cos(k\theta),\qquad
A^{-1}\phi=\zeta\cos(n\theta),\notag
\end{align}
hence
 \begin{align}\label{Poiss.cc-mn}
 & \clB(\psi) \phi=\{\eta\cos(k\theta),\phi\}_\fkp - \{\psi,\zeta\cos(n\theta)\}_\fkp.
  \end{align}   
Using~\eqref{poisson_b-pol}, with~$\varpi_{(i,j)}\coloneqq\sin(i\theta)\cos(j\theta)$, we find
\begin{align}
\{\eta\cos(k\theta),\phi\}_\fkp&= \tfrac1{r}\left(\tfrac{\p\eta}{\p r}\tfrac{\p\phi}{\p\theta}\cos(k\theta)+k\eta\tfrac{\p\phi}{\p r}\sin(k\theta)\right)\notag\\
&= -\tfrac{n}{r}\left((2k+2)r^{2k+1}-(2k+4)r^{2k+3}\right)r^{2n}(\xi_{(m+1,n)}r^{2m}-\xi_{(1,n)})\varpi_{(n,k)}\notag\\
&\quad+\tfrac{k}{r}\left(r^{2k+2}-r^{2k+4}\right)r^{2n-1}(2(n+m)\xi_{(m+1,n)}r^{2m}-2n\xi_{(1,n)})\varpi_{(k,n)}.\notag
  \end{align}  
 Hence,  using~$2\varpi_{(n,k)}=\sin((n+k)\theta) +\sin((n-k)\theta)$, we arrive at
 \begin{subequations}\label{Poiss.cc-mn-1}
\begin{align}
\{\eta\cos(k\theta),\phi\}_\fkp&=r^{2(n+k)} \underline{\rho_s}\sin((n-k)\theta)+r^{2(n+k)} \overline{\rho_s}\sin((n+k)\theta),
  \end{align} 
 with the factors
 \begin{align}
 \underline{\rho_s}&\coloneqq -n\left((k+1)-(k+2)r^{2}\right)(\xi_{(m+1,n)}r^{2m}-\xi_{(1,n)})\notag\\
&\;\quad-k\left(1-r^{2}\right)((n+m)\xi_{(m+1,n)}r^{2m}-n\xi_{(1,n)});\notag\\
 \overline{\rho_s}&\coloneqq -n\left((k+1)-(k+2)r^{2}\right)(\xi_{(m+1,n)}r^{2m}-\xi_{(1,n)})\notag\\
&\;\quad+k\left(1-r^{2}\right)((n+m)\xi_{(m+1,n)}r^{2m}-n\xi_{(1,n)}).
  \end{align}  
   \end{subequations}

For the second term in the right-hand side of~\eqref{Poiss.cc-mn} we find
\begin{align}
&\{\psi,\zeta\cos(n\theta)\}_\fkp= \tfrac1{r}\left(-n\tfrac{\p\psi}{\p r}\zeta\sin(n\theta)-\tfrac{\p\psi}{\p\theta}\tfrac{\p\zeta}{\p r}\cos(n\theta)\right)\notag\\
&=-\tfrac{n}{r}\left(2(k+1)r^{2k+1}\xi_{(2,k)}-2kr^{2k-1}\xi_{(1,k)}\right)r^{2n}(r^{2}-r^{2(m+1)})\varpi_{(n,k)} \notag\\
&\quad+\tfrac{k}{r}\left(\xi_{(2,k)}r^{2k+2}-\xi_{(1,k)}r^{2k}\right)r^{2n+1}\bigl(2(n+1)-2(m+1+n)r^{2m}\bigr)\varpi_{(k,n)}.\notag
  \end{align}  
Hence, 
it follows that
 \begin{subequations}\label{Poiss.cc-mn-2}
\begin{align}
\{\psi,\zeta\cos(n\theta)\}_\fkp&=r^{2(n+k)} \underline{\sigma_s}\sin((n-k)\theta)+r^{2(n+k)} \overline{\sigma_s}\sin((n+k)\theta),
  \end{align} 
 with
 \begin{align}
 \underline{\sigma_s}&\coloneqq-n\left((k+1)r^{2}\xi_{(2,k)}-k\xi_{(1,k)}\right)(1-r^{2m}) \notag\\
&\quad-k\left(\xi_{(2,k)}r^{2}-\xi_{(1,k)}\right)\bigl((n+1)-(m+1+n)r^{2m}\bigr)\notag;\\
\overline{\sigma_s}&\coloneqq -n\left((k+1)r^{2}\xi_{(2,k)}-k\xi_{(1,k)}\right)(1-r^{2m}) \notag\\
&\quad+k\left(\xi_{(2,k)}r^{2}-\xi_{(1,k)}\right)\bigl((n+1)-(m+1+n)r^{2m}\bigr).
  \end{align} 
\end{subequations}

By combining~\eqref{Poiss.cc-mn}, \eqref{Poiss.cc-mn-1}, and~\eqref{Poiss.cc-mn-2}, we obtain
 \begin{subequations}\label{Poiss.cc-Poly}
 \begin{align}
 & \clB(\psi) \phi=r^{2(n+1)}(\underline{\rho_s}-\underline{\sigma_s})\sin((n-k)\theta)+r^{2(n+1)} (\overline{\rho_s}-\overline{\sigma_s})\sin((n+k).
  \end{align}   
Finally, note that
 \begin{align}
  \overline{\rho_s}-\overline{\sigma_s} &={C}_{+,1,m}^{(n,k)}r^{2(m+1)}+C_{+,0,m}^{(n,k)}r^{2m}+C_{+,1}^{(n,k)}r^2+C_{+,0}^{(n,k)},\\
  \underline{\rho_s}-\underline{\sigma_s} &={C}_{-,1,m}^{(n,k)}r^{2(m+1)}+C_{-,0,m}^{(n,k)}r^{2m}+C_{-,1}^{(n,k)}r^2+C_{-,0}^{(n,k)},  
     \end{align}  
  \end{subequations}   
with
\begin{align}\notag
{C}_{+,1,m}^{(n,k)}&\coloneqq n(k+2)\xi_{(m+1,n)}-k(n+m)\xi_{(m+1,n)}-n(k+1)\xi_{(2,k)}+k(m+1+n)\xi_{(2,k)}\notag\\
&= (2n-km)\xi_{(m+1,n)}+(k(m+1)-n)\xi_{(2,k)},\notag\\
C_{+,0,m}^{(n,k)}&\coloneqq -n(k+1)\xi_{(m+1,n)}+k(n+m)\xi_{(m+1,n)}+nk\xi_{(1,k)}-k(m+1+n)\xi_{(1,k)}\notag\\
&= (km-n)\xi_{(m+1,n)}-k(m+1)\xi_{(1,k)},\notag\\
C_{+,1}^{(n,k)}&\coloneqq-n(k+2)\xi_{(1,n)}+kn\xi_{(1,n)}+n(k+1)\xi_{(2,k)}-k(n+1)\xi_{(2,k)}\notag\\
&=-2n\xi_{(1,n)}+(n-k)\xi_{(2,k)},\notag\\
C_{+,0}^{(n,k)}&\coloneqq n(k+1)\xi_{(1,n)}- kn\xi_{(1,n)}-nk\xi_{(1,k)}+k(n+1)k\xi_{(1,k)}\notag\\
&=n\xi_{(1,n)}+k\xi_{(1,k)},\notag
\intertext{and}
{C}_{-,1,m}^{(n,k)}&\coloneqq n(k+2)\xi_{(m+1,n)}+k(n+m)\xi_{(m+1,n)}-n(k+1)\xi_{(2,k)}-k(m+1+n)\xi_{(2,k)}\notag\\
&= (k(m+2n)+2n)\xi_{(m+1,n)}-(k(m+1+2n)+n)\xi_{(2,k)},\notag\\
C_{-,0,m}^{(n,k)}&\coloneqq -n(k+1)\xi_{(m+1,n)}-k(n+m)\xi_{(m+1,n)}+nk\xi_{(1,k)}+k(m+1+n)\xi_{(1,k)}\notag\\
&= -(k(m+2n)+n)\xi_{(m+1,n)}+k(m+1+2n)\xi_{(1,k)},\notag\\
C_{-,1}^{(n,k)}&\coloneqq-n(k+2)\xi_{(1,n)}-kn\xi_{(1,n)}-n(k+1)\xi_{(2,k)}-k(n+1)\xi_{(2,k)}\notag\\
&=-2n(k+1)\xi_{(1,n)}+(k(2n+1)+n)\xi_{(2,k)},\notag\\
C_{-,0}^{(n,k)}&\coloneqq n(k+1)\xi_{(1,n)}+kn\xi_{(1,n)}-nk\xi_{(1,k)}-k(n+1)k\xi_{(1,k)}\notag\\
&=n(2k+1)\xi_{(1,n)}-k(2n+1)\xi_{(1,k)},\notag
  \end{align}  
which finishes the proof.
\qed

\subsection{Proof of Lemma~\ref{L:poiss(n+k)_cs_gen}}\label{Apx:proofL:poiss(n+k)_cs_gen}

We follow a slight variation of the arguments in Section~\ref{Apx:proofL:poiss(n+k)_cc_gen}.

Let us denote, for positive integers~$n\ge1$ and~$k\ge1$,
\begin{align}
\psi&\coloneqq\xi_{(2,k)}\fkc_{(1,k)}-\xi_{(1,k)}\fkc_{(0,k)},\qquad &\eta&\coloneqq r^{2k}(r^{2}-r^{4})\notag\\
\varphi&\coloneqq\xi_{(m+1,n)}\fks_{(m,n)}-\xi_{(1,n)}\fks_{(0,n)},\qquad &\zeta&\coloneqq r^{2n}(r^{2}-r^{2(m+1)}).\notag
\end{align}

By Corollary~\ref{C:invlap_cs} we have that
\begin{align}
A^{-1}\psi&=\eta\cos(k\theta),\qquad
A^{-1}\varphi=\zeta\sin(n\theta),\notag
\end{align}
hence
 \begin{align}\label{Poiss.cs-mn}
 & \clB(\psi) \varphi=\{\eta\cos(k\theta),\varphi\}_\fkp - \{\psi,\zeta\sin(n\theta)\}_\fkp.
  \end{align}   
Using~\eqref{poisson_b-pol}, with~$\gamma^s_{(i,j)}\coloneqq\sin(i\theta)\sin(j\theta)$ and~$\gamma^c_{(i,j)}\coloneqq\cos(i\theta)\cos(j\theta)$, we find
\begin{align}
\{\eta\cos(k\theta),\varphi\}_\fkp&= \tfrac1{r}\left(\tfrac{\p\eta}{\p r}\tfrac{\p\varphi}{\p\theta}\cos(k\theta)+k\eta\tfrac{\p\varphi}{\p r}\sin(k\theta)\right)\notag\\
&= \tfrac{n}{r}\left((2k+2)r^{2k+1}-(2k+4)r^{2k+3}\right)r^{2n}(\xi_{(m+1,n)}r^{2m}-\xi_{(1,n)})\gamma^c_{(n,k)}\notag\\
&\quad+\tfrac{k}{r}\left(r^{2k+2}-r^{2k+4}\right)r^{2n-1}(2(n+m)\xi_{(m+1,n)}r^{2m}-2n\xi_{(1,n)})\gamma^s_{(n,k)}.\notag
  \end{align}  
Now, by~$2\gamma^c_{(n,k)}=\cos((n+k)\theta) +\cos((n-k)\theta)$ and~$2\gamma^s_{(n,k)}=\cos((n-k)\theta) -\cos((n+k)\theta)$,
 \begin{subequations}\label{Poiss.cs-mn-1}
\begin{align}
\{\eta\cos(k\theta),\varphi\}_\fkp&=r^{2(n+k)} \underline{\rho_c}\cos((n-k)\theta)+r^{2(n+k)} \overline{\rho_c}\cos((n+k)\theta),
  \end{align} 
 with the factors
 \begin{align}
 \underline{\rho_c}&\coloneqq +n\left((k+1)-(k+2)r^{2}\right)(\xi_{(m+1,n)}r^{2m}-\xi_{(1,n)})\notag\\
&\;\quad +k\left(1-r^{2}\right)((n+m)\xi_{(m+1,n)}r^{2m}-n\xi_{(1,n)});\notag\\
 \overline{\rho_c}&\coloneqq +n\left((k+1)-(k+2)r^{2}\right)(\xi_{(m+1,n)}r^{2m}-\xi_{(1,n)})\notag\\
&\;\quad-k\left(1-r^{2}\right)((n+m)\xi_{(m+1,n)}r^{2m}-n\xi_{(1,n)}).
  \end{align}  
   \end{subequations}

For the second term in the right-hand side of~\eqref{Poiss.cs-mn} we find
\begin{align}
&\{\psi,\zeta\sin(n\theta)\}_\fkp= \tfrac1{r}\left(n\tfrac{\p\psi}{\p r}\zeta\cos(n\theta)-\tfrac{\p\psi}{\p\theta}\tfrac{\p\zeta}{\p r}\sin(n\theta)\right)\notag\\
&=+\tfrac{n}{r}\left(2(k+1)r^{2k+1}\xi_{(2,k)}-2kr^{2k-1}\xi_{(1,k)}\right)r^{2n}(r^{2}-r^{2(m+1)})\gamma^c_{(n,k)} \notag\\
&\quad+\tfrac{k}{r}\left(\xi_{(2,k)}r^{2k+2}-\xi_{(1,k)}r^{2k}\right)r^{2n+1}\bigl(2(n+1)-2(m+1+n)r^{2m}\bigr)\gamma^s_{(n,k)}.\notag
  \end{align}  
Hence, 
it follows that
 \begin{subequations}\label{Poiss.cs-mn-2}
\begin{align}
\{\psi,\zeta\sin(n\theta)\}_\fkp&=r^{2(n+k)} \underline{\sigma_c}\cos((n-k)\theta)+r^{2(n+k)} \overline{\sigma_c}\cos((n+k)\theta),
  \end{align} 
 with
 \begin{align}
 \underline{\sigma_c}&\coloneqq+n\left((k+1)r^{2}\xi_{(2,k)}-k\xi_{(1,k)}\right)(1-r^{2m}) \notag\\
&\quad+k\left(\xi_{(2,k)}r^{2}-\xi_{(1,k)}\right)\bigl((n+1)-(m+1+n)r^{2m}\bigr)\notag;\\
\overline{\sigma_c}&\coloneqq +n\left((k+1)r^{2}\xi_{(2,k)}-k\xi_{(1,k)}\right)(1-r^{2m}) \notag\\
&\quad-k\left(\xi_{(2,k)}r^{2}-\xi_{(1,k)}\right)\bigl((n+1)-(m+1+n)r^{2m}\bigr)
  \end{align} 
\end{subequations}

By combining~\eqref{Poiss.cs-mn}, \eqref{Poiss.cs-mn-1}, and~\eqref{Poiss.cs-mn-2}, we obtain
 \begin{align}\notag
 & \clB(\psi) \phi=r^{2(n+1)}(\underline{\rho_c}-\underline{\sigma_c})\sin((n-k)\theta)+r^{2(n+1)} (\overline{\rho_c}-\overline{\sigma_c})\sin((n+k).
  \end{align} 
  
 Note that  the expressions in~\eqref{Poiss.cs-mn-1} and~\eqref{Poiss.cs-mn-2} related to those in~\eqref{Poiss.cc-mn-1} and~\eqref{Poiss.cc-mn-2} by
   \begin{align}\notag
   \underline{\rho_c}=- \underline{\rho_s},\qquad \underline{\sigma_c}=-\underline{\sigma_s},\qquad
    \overline{\rho_c}=-\overline{\rho_s},\qquad \overline{\sigma_c} =-\overline{\sigma_s}.
    \end{align} 
  
Therefore, we arrive at the analogue of~\eqref{Poiss.cc-Poly},
 \begin{align}\notag
 & \clB(\psi) \phi=-r^{2(n+1)}Q_{-,m}^{(n,k)}(r^2)\cos((n-k)\theta)-r^{2(n+1)}Q_{+,m}^{(n,k)}(r^2) \cos((n+k),
  \end{align}   
 with the polynomials~$Q_{+,m}^{(n,k)}$ and~$Q_{-,m}^{(n,k)}$ as in Lemma~\ref{L:poiss(n+k)_cc_gen}.
\qed

\subsection{Proof of Lemma~\ref{L:3vec-(n+1)}}\label{Apx:proofL:3vec-(n+1)}
Recall that, by Lemma~\ref{L:poiss(n+k)_cc_gen}, we have
\begin{align}\notag
Q^{(n,k)}_{+,m}(s)&={C}_{+,1,m}^{(n,k)}s^{m+1}+C_{+,0,m}^{(n,k)}s^{m}+C_{+,1}^{(n,k)}s+C_{+,0}^{(n,k)},\\
\intertext{In particular,}
Q^{(n,1)}_{+,1}(s)&={C}_{+,1,1}^{(n,1)}s^{2}+C_{+,0,1}^{(n,1)}s+C_{+,1}^{(n,1)}s+C_{+,0}^{(n,1)},\notag \\
Q^{(n-1,2)}_{+,1}(s)&={C}_{+,1,1}^{(n-1,2)}s^{2}+C_{+,0,1}^{(n-1,2)}s+C_{+,1}^{(n-1,2)}s+C_{+,0}^{(n-1,2)},\notag \\
Q^{(n,1)}_{+,2}(s)&={C}_{+,1,2}^{(n,1)}s^{3}+C_{+,0,2}^{(n,1)}s^{2}+C_{+,1}^{(n,1)}s+C_{+,0}^{(n,1)},\notag \\
Q^{(n-1,2)}_{+,2}(s)&={C}_{+,1,2}^{(n-1,2)}s^{3}+C_{+,0,2}^{(n-1,2)}s^{2}+C_{+,1}^{(n-1,2)}s+C_{+,0}^{(n-1,2)}.\notag 
   \end{align} 
 
 Now we introduce the matrix
  \begin{align}\notag  
X^{[n]}\coloneqq
 \begin{bmatrix}
 C_{+,0}^{(n,1)}&\quad C_{+,0}^{(n-1,2)} &\quad C_{+,0}^{(n,1)} &\quad C_{+,0}^{(n-1,2)} \\
C_{+,0,1}^{(n,1)}+C_{+,1}^{(n,1)}&\quad C_{+,0,1}^{(n-1,2)}+C_{+,1}^{(n-1,2)}&\quad C_{+,1}^{(n,1)} &\quad C_{+,1}^{(n-1,2)} \\
{C}_{+,1,1}^{(n,1)} &\quad {C}_{+,1,1}^{(n-1,2)}&\quad C_{+,0,2}^{(n,1)}&\quad C_{+,0,2}^{(n-1,2)}\\
0&\quad0 &\quad {C}_{+,1,2}^{(n,1)}&\quad {C}_{+,1,2}^{(n-1,2)}
 \end{bmatrix} 
   \end{align}    
  and observe that 
  \begin{align}\notag  
\linspan\{Q^{(n,1)}_{+,1}(s),Q^{(n-1,2)}_{+,1}(s),Q^{(n,1)}_{+,2}(s),Q^{(n-1,2)}_{+,2}(s)\}=\begin{bmatrix}
 1&s&s^2&s^3
 \end{bmatrix} X^{[n]} \bbR^{4\times 1}.
   \end{align}   
 Therefore, the statement of Lemma~\ref{L:3vec-(n+1)} will follow, once we show that~$X^{[n]}\bbR^{4\times 1}=\bbR^{4\times 1}$,
that is once we show that~$X^{[n]}$ is nonsingular, for integers~$n\ge3$. For this purpose, we can find  an expression for~$\det(X^{[a]})$ by symbolic computations using~{\sc Mathematica}
~(V.~13.2, Wolfram Research Inc., Champaign, IL, USA, 2023, \url{https://www.wolfram.com/mathematica}),
replacing~$n$ by a real variable~$a\in\bbR$. In this way we find the polynomial expression
\begin{subequations}\label{detXs}
\begin{align}
\det(X^{[a]})&\coloneqq
-8 (a-2)(a+1)^2 \clP(a) ,\qquad a\in\bbR,\\
\mbox{with}\quad\clP(a)&\coloneqq-405405 - 257382a + 37892a^2 + 64702a^3\notag\\
&\quad\; + 19149a^4 + 2376a^5 + 108a^6.
 \end{align}
 \end{subequations}

Next we show that~$X^{[a]}$ is nonsingular for all~$a>2$. 
It is sufficient to show that the polynomial~$\clP(a)$ in~\eqref{detXs} has no zeros for~$a>2$.
Observe that
 \begin{equation}\label{ddetXn6}
(\tfrac{\rmd}{\rmd a})^6\clP(a)=108\cdot 6!>0,\quad\mbox{for all}\quad a>0,
 \end{equation}
 and, again by symbolic computations using~{\sc Mathematica}, we can find that
  \begin{align} 
   (\tfrac{\rmd}{\rmd a})^5\clP(2)&=440640,\qquad&
   (\tfrac{\rmd}{\rmd a})^4\clP(2)&=1185336,\notag\\
   (\tfrac{\rmd}{\rmd a})^3\clP(2)&=1981284,\qquad&
   (\tfrac{\rmd}{\rmd a})^2\clP(2)&=2203360,\notag\\
   \tfrac{\rmd}{\rmd a}\clP(2)&=1494194,\qquad&
   \clP(2)&=138343,\notag
   \end{align}  
 which combined with~\eqref{ddetXn6} imply that~$\clP(a)>0$ for all~$a>2$. Note that~$(\tfrac{\rmd}{\rmd a})^{i+1}>0$ for all~$a>2$ and~$(\tfrac{\rmd}{\rmd a})^{i}\clP(2)>0$ imply that~$(\tfrac{\rmd}{\rmd a})^{i}\clP(a)>0$ for all~$a>2$.
Further, we can also find that $\clP(1)=-538560$,
which leads us to
~$\det(X^{[n]})\ne0$ for all~$n\in\bbN_+\setminus\{2\}$.
\qed

\subsection{Proof of Lemma~\ref{L:n=k}}\label{Apx:proofL:n=k}
Recalling ${C}_{-,1,m}^{(n,k)}$ in Lemma~\ref{L:poiss(n+k)_cc_gen}, for $k=n\ge1$, we find
\begin{align}
{C}_{-,1,m}^{(n,k)}&=  (k(m+2n)+2n)\xi_{(m+1,n)}-(k(m+1+2n)+n)\xi_{(2,k)}\notag\\
{C}_{-,1,m}^{(n,n)}&= (n(m+2n)+2n)\xi_{(m+1,n)}-(n(m+1+2n)+n)\xi_{(2,n)}\notag\\
&= n(m+2n +2)(\xi_{(m+1,n)}-\xi_{(2,n)})\notag
\end{align}
and, recalling~\eqref{xi_mn},
\begin{align}
\tfrac{1}{n(m+2n +2)}{C}_{-,1,m}^{(n,k)}&=4(m+1+n)^2-n^2-(4(n+2)^2-n^2)\notag\\
&=4((m-1)^2+2(m-1)(n+2))=4(m-1)(m+3+2n),\notag
\end{align}
which gives us~${C}_{-,1,m}^{(n,k)}>0$ for all~$m>1$.
\qed

\bigskip\noindent
{\bf Aknowlegments.}
S. Rodrigues acknowledges partial support from
the State of Upper Austria and the Austrian Science
Fund (FWF): P 33432-NBL.
%

%
%

\bibliographystyle{plainurl}
{\smaller
\bibliography{AppConNS_Disc}
}
\end{document}